\documentclass[letterpaper,11pt]{article}

\usepackage[utf8]{inputenc}
\usepackage[margin=1.15in]{geometry}
\usepackage{amssymb,amsmath,amsthm,mathtools,amsbsy,amsfonts}
\usepackage{subcaption}
\usepackage{graphicx,wasysym}
\usepackage{stmaryrd}
\usepackage{cases}
\usepackage{empheq}
\usepackage{paralist}
\usepackage{hyperref}
\usepackage{multimedia}
\usepackage{tikz}
\usepackage{enumitem}
\usepackage{cite}
\usepackage{multirow}
\usepackage[ruled,vlined]{algorithm2e}
\SetKwRepeat{Do}{do}{while}
\usepackage{threeparttable}

\graphicspath{{./fig/}}

\usepackage{fullpage}

\usepackage[explicit]{titlesec}
\titleformat{\section}{\large\bfseries\center\raggedright}{\thesection}{0.5em}{{#1}}[]
\titleformat{\subsection}{\bfseries\center\raggedright}{\thesubsection}{0.5em}{{#1}}[]
\titleformat{\subsubsection}[runin]{\bfseries}{\thesubsubsection}{0.5em}{{#1}}[.]
\titlespacing*{\section}{0pt}{0.8\baselineskip}{0.6\baselineskip}
\titlespacing*{\subsection}{0pt}{0.6\baselineskip}{0.4\baselineskip}
\titlespacing*{\subsubsection}{0pt}{0.4\baselineskip}{0.4\baselineskip}

\newcommand\namefont{\normalfont\bfseries}
\newcommand\numberfont{\normalfont\bfseries}
\newcommand\notefont{\normalfont\bfseries}
\newtheoremstyle{mystyle} 
    {0.3em} 
    {0.3em} 
    {\itshape} 
    {} 
    {\normalfont} 
    {.} 
    {.5em} 
    {{\namefont\thmname{#1}}~{\numberfont\thmnumber{#2}}{\notefont\thmnote{ (#3)}}} 
\theoremstyle{plain}
\newtheorem{thm}{Theorem}[section]

\newtheorem{rem}[thm]{Remark}
\newtheorem{lem}[thm]{Lemma}

\newtheorem{defn}[thm]{Definition}

\allowdisplaybreaks
\mathtoolsset{showonlyrefs}
\renewcommand{\leq}{\leqslant}
\renewcommand{\geq}{\geqslant}

\newcommand{\ul}[1]{\underline{\smash{#1}}}

\definecolor{darkred}{rgb}{0.6,0.1,0.1}
\definecolor{darkgreen}{rgb}{0.1,0.6,0.1}
\definecolor{darkblue}{rgb}{0.1,0.1,0.6}

\newcommand{\mt}[1]{\mathrm{#1}}
\newcommand{\overbar}[1]{\mkern 1.7mu\overline{\mkern-3.7mu#1\mkern-1.7mu}\mkern 1.7mu} 
\def\st{\, \left|\right. \,} 
\def\:{\colon} 

\newcommand{\abs}[1]{\left\vert#1\right\vert}
\newcommand{\ap}[1]{\left\langle#1\right\rangle} 
\newcommand{\norm}[1]{\left\Vert#1\right\Vert}

\def\grad{\nabla}
\def\bgrad{\boldsymbol{\nabla}}

\def\bLap{\boldsymbol{\Delta}}

\DeclareMathOperator{\dive}{div}

\DeclareMathOperator{\dist}{dist}

\DeclareMathOperator{\dom}{dom} 

\def\R{\mathbb{R}} 
\def\N{\mathbb{N}} 

\def\d{\,\mathrm{d}}

\def\p{\partial}

\newcommand{\bs}[1]{\boldsymbol{#1}}

\def\bnull{\boldsymbol{0}}
\def\bx{\boldsymbol{x}}

\def\bu{\boldsymbol{u}}
\def\bv{\boldsymbol{v}}
\def\bw{\boldsymbol{w}}

\def\bg{\boldsymbol{g}}
\def\bff{\boldsymbol{f}}
\def\bn{\boldsymbol{n}}

\def\bK{\boldsymbol{K}}

\def\bI{\boldsymbol{I}}
\def\bvarphi{\boldsymbol{\varphi}}
\def\bLambda{\boldsymbol{\Lambda}}
\def\blambda{\boldsymbol{\lambda}}

\def\bL{\boldsymbol{L}}
\def\bV{\boldsymbol{V}}
\def\bB{\boldsymbol{B}}

\def\F{\mathcal{F}} 
\def\E{\mathcal{E}} 

\def\D{\mathcal{D}} 

\begin{document}

\title{Model adaptation in a discrete fracture network:\\ existence of solutions and
numerical strategies}

\author{Alessio Fumagalli$^1$ \and Francesco Saverio Patacchini$^2$}

\date{
$^1$ Department of Mathematics, Politecnico di Milano, p.za Leonardo da Vinci 32, Milano 20133, Italy\\%
$^2$ IFP Energies nouvelles, 1 et 4 avenue de Bois-Pr\'eau, 92852 Rueil-Malmaison, France
}

\maketitle

\begin{abstract}
    Fractures are normally present in the underground and are, for some physical
    processes, of paramount importance. Their accurate description is 
    fundamental to obtain reliable numerical outcomes useful, e.g., for energy
    management. Depending on the physical and geometrical properties of the
    fractures, fluid flow can behave differently, going from a slow Darcian
    regime to more complicated Brinkman or even Forchheimer regimes for high velocity.
    The main problem is to determine where in the fractures one regime is more adequate
    than others. In order to determine these low-speed and high-speed regions, this 
    work proposes an adaptive strategy which is based on selecting the appropriate 
    constitutive law linking velocity and pressure according to a threshold criterion
    on the magnitude of the fluid velocity itself. Both theoretical and numerical 
    aspects are considered and investigated, showing the potentiality of the 
    proposed approach. From the analytical viewpoint, we show existence of
    weak solutions to such model under reasonable hypotheses on the constitutive laws.
    To this end, we use a variational approach identifying solutions with minimizers of
    an underlying energy functional. From the numerical viewpoint, we propose a
    one-dimensional algorithm which tracks the interface between the low- and high-speed
    regions. By running numerical experiments using this algorithm, we illustrate 
    some interesting behaviors of our adaptive model on a single fracture and small
    networks of intersecting fractures.
\end{abstract}


\section{Introduction}\label{sec:intro}

Fractures are discontinuities, often assumed planar, 
along which a rock has been broken due to a pre-existing stress state, and represent 
the main conduits for fluid flow. Because of subsequent mineralization by chemical 
reactions along the fracture walls, a fracture may be partially or completely filled 
by material, substantially altering its physical properties and changing the fluid 
circulation. Additional rock deformation may change even more
the hydraulic properties of the fractures. Even if there is not a clear
separation of scales for fractures, since generally they span all sizes,
we consider here only fractures at a certain scale leaving the smaller ones to be
part of the rock matrix. A fracture's size has an impact on its aperture and 
consequently on its flow response.

Due to their geometrical complexity, fractures are normally represented as
objects of co-dimension 1, meaning that in a $d$-dimensional porous medium they
are approximated as $(d-1)$-surfaces. This approach is normally referred to as
discrete fracture model (DFM), see for example
\cite{Martin2005,DAngelo2011,Schwenck2015,Brenner2016a,Flemisch2016a,Berre2019b,Facciola2019,Berre2020a},
where the physical processes involved are written as a set of partial
differential equations with interface conditions between rock matrix and fractures. The resulting system is
mixed-dimensional partial differential problem as discussed in
\cite{Frih2011,Boon2018,Nordbotten2018,Fumagalli2017a,Boon2020}. For some
specific problems the contribution of the rock matrix is negligible and can be
omitted, obtaining the so-called discrete fracture network (DFN) models. On this topic, the
reader may refer to
\cite{Erhel2009,Berrone2013,Dreuzy2013,Berrone2014,Berrone2016,Fumagalli2016a,Benedetto2017,Scialo2017,Borio2019a}
and the references therein. Especially in the presence of complex constitutive equations
in the fractures, the authors in \cite{Ahmed2018,Ahmed2019} showed that it is
possible to separate the contribution of the porous medium from the fracture
network via a Robin-to-Neumann operator.  They further showed that in a problem
where fractures are modeled with a non-linear problem while the surrounding
porous medium still obeys a linear law, it is possible to avoid the
computationally expensive porous medium in the non-linear iterations considering
only the fracture network.  In the present work, we will consider only fracture networks
while keeping in mind the possibility to include the surrounding porous medium via
this approach.

Depending on several aspects, mainly the micro-structure and hydraulic aperture
of a fracture, the flow in the fracture can be classified into different regimes
corresponding to increasing flow rates and corresponding increasing mathematical
and numerical difficulties. In many applications for low flow rates, caused by
a combination of filling materials in the fracture and a relatively low-pressure
gradient, Darcy flow can be considered in the fracture and a reduced model of
Darcy-type can be derived, see among others
\cite{Martin2005,Angot2009,Frih2011,DAngelo2011,Formaggia2012,Schwenck2015,Flemisch2016a},
in the lower-dimensional setting.  Most of the research so far has been focused on this
flow model; however, its validity is questionable and, to a large extend, it is
insufficient for real problems. For increased flow rates, for example when a
fracture is an open and narrow channel or the packaging of grains is too coarse,
viscous effects become important and a Brinkman or Stokes equation is more
coherent as a reduced model to describe the flow; see
\cite{Morales2017,Rybak2020}. Finally for high velocities, because of inertial
effects, experiments show deviations from the previous models, which indicate
the need for a non-linear correction term. The authors in \cite{Frih2008,Knabner2014} proposed
a reduced model based on a Darcy--Forchheimer flow to capture this phenomenon. This
effect is more evident in large objects like faults spanning several hundreds of
meters. In all the lower-dimensional models, the flow in the surrounding rock
matrix is still modeled by Darcy's law and proper interface conditions are used to
couple matrix and fracture flow.

Depending on their nature, fracture networks may exhibit all flow regimes in
different regions separated by transition interfaces and coupled by suitable
conditions. The interfaces can be located using Forchheimer and Reynolds numbers, 
see \cite{Zeng2006}, which depend on the water velocity. This is the setting for
the present paper, where the positions of the interfaces are
not fixed at the outset of the problem: we obtain a (multi-physics) non-linear
free-boundary problem on the fracture network. We present a mathematical framework 
that is able to adapt the constitutive law in accordance with the flow regime. 
As a simplification, we assume that only two laws, i.e., two speed regimes, can 
be prescribed, leaving the case of multiple laws as a future work. 
A theoretical analysis is developed along with a numerical algorithm
that tracks the interface. Theoretically, we show existence of solutions to our
problem via the minimization of an underlying energy. We do not show uniqueness
since we make the choice to let the constitutive law linking 
velocity and pressure on the interface be a free parameter. We also show that the convex
nature of the problem is strongly intertwined with the "direction" of the jump
between the high-speed and low-speed laws; indeed, if the permeability increases
from low- to high-speed the problem is non-convex, otherwise it is convex. In the former
case we are forced to restrict our existence result to one space dimension ($d=1$).
Numerically, the examples we give with one and more fractures illustrate the
quality and applicability of the proposed algorithm, especially in the non-convex case. 
In the convex case, the algorithm often features oscillations between configurations which
prevents it from converging. For the numerical simulations we used the 
library PorePy \cite{Keilegavlen2020}, a simulation tool written in Python for fractured 
and deformable porous media. PorePy is freely available on GitHub along with the numerical
tests proposed in this work.

The paper is organized as follows. In Section \ref{sec:proposed-model} we decribe 
our model, introduce the equations as well as some notation. In Section \ref{sec:math-setting}
we give the rigorous mathematical setting, along with the assumptions on the
constitutive laws and the weak formulation of the problem. Section \ref{sec:existence}
is dedicated to our results on the existence of solutions and their proofs. Section 
\ref{sec:numerical_approx} introduces the discrete formulation of the
problem and a suitable algorithm to solve it. Numerical examples are reported
in Section \ref{sec:examples} for increasing geometrical and physical complexity.
The work finishes with conclusions in Section \ref{sec:conclusion}.


\section{Proposed model}
\label{sec:proposed-model}

We focus here on a single fracture; see Section \ref{subsec:intersecting-fractures} for a discussion on the model for multiple intersecting fractures. We identify our fracture with an open, bounded connected set $\Omega \subset \R^d$ with Lipschitz boundary $\p\Omega$. We suppose $\Omega$ is filled with a fluid of constant density $\rho \equiv 1$. Mass conservation in $\Omega$ then reads
\begin{equation}
  \label{eq:div-q}
  \dive \bu = q,
\end{equation} 
where $\bu\: \Omega \to \R^d$ is the unknown velocity of the fluid and $q\: \Omega \to \R$ a given source term allowing, for instance, for fluid mass to be exchanged between the fractures and the surrounding porous medium, or rock matrix. Let $\Sigma_{\mt{v}},\Sigma_{\mt{p}} \subset \p\Omega$ be relatively open and such that $\p\Omega = \overbar{\Sigma_{\mt{v}} \cup \Sigma_{\mt{p}}}$ and $\Sigma_{\mt{v}} \cap \Sigma_{\mt{p}} = \emptyset$. To \eqref{eq:div-q} we add the following boundary conditions:
\begin{equation}
  \label{eq:neumann-cond}
  \begin{cases}
    \bu \cdot \bn = u_0 & \text{on $\Sigma_{\mt{v}}$},\\
    p = p_0 & \text{on $\Sigma_{\mt{p}}$},
  \end{cases}
\end{equation}
where $p\:\Omega \to \R$ is the unknown pressure of the fluid and $\bn$ is the outward normal unit vector of $\p\Omega$. Here, $u_0: \Sigma_{\mt{v}} \rightarrow \R$ and $p_0\:\Sigma_{\mt{p}}\to\R$ are given functions setting the conditions on the boundary on $\bu$ and $p$, respectively. We are denoting maps on $\Omega$ and their traces on the boundary $\p\Omega$ by the same notation.  

\subsection{Velocity-pressure constitutive law}
\label{sec:vel-press-law}

Typically, one couples \eqref{eq:div-q} and \eqref{eq:neumann-cond} with a constitutive relation, or law, between the velocity field $\bu$ and the pressure field $p$ via some operator $\bLambda$:
\begin{equation}
  \label{eq:relation-u-p}
  \bLambda(\bu) = -\bgrad p + \bff,
\end{equation}
where $\bff$ is a given body force (e.g., gravity). Examples of laws $\bLambda$ relating velocity and pressure via \eqref{eq:relation-u-p} are
\begin{equation}\label{eq:F-ex}
  \begin{gathered}
    \bLambda_\mt{S}(\bu) = -\bLap \bu, \quad \bLambda_\mt{D}(\bu) = \bK^{-1} \bu,\\
    \quad \bLambda_\mt{B}(\bu) = -\bLap \bu + \bK^{-1} \bu, \quad \bLambda_\mt{DF}(\bu) = \bK^{-1} \bu + \alpha\norm{\bu}^{r-2}\bu,
  \end{gathered} 
\end{equation}
where $\bK \: \Omega \to \R^{d\times d}$ is the permeability tensor and $\alpha\geq 0$ and $r\in[2,\infty)$ some parameters. The notation $\norm{\cdot}$ stands for the Euclidean norm on $\R^d$. Choosing $\bLambda_\mt{S}$ gives Stokes' equation, $\bLambda_\mt{D}$ gives Darcy's equation, $\bLambda_\mt{B}$ gives Brinkman's equation and $\bLambda_\mt{DF}$ gives the generalized Darcy--Forchheimer equation (which simplifies into the classical Darcy--Forchheimer when $r=3$).

To the authors' knowledge, a known issue that has not yet found a documented answer is the case when one needs to choose a combination of laws such as those given as examples in~\eqref{eq:F-ex}, rather than a single one, i.e., when one needs to couple different velocity-pressure laws according to some validity criterion which selects the better-adapted law. As already mentioned in the introduction, this criterion should depend on the speed regime (or Reynolds number) of the flow. For example, where the Reynolds number is low Darcy's law $\bLambda_\mt{D}$ may be preferred, whereas where it is high the Darcy--Forchheimer law $\bLambda_\mt{S}$ might be chosen. For this reason, in this paper we consider the case where we need to choose from two laws $\bLambda_1$ and $\bLambda_2$ according to some threshold speed $\bar u > 0$. We expect that generalizing our results to more than two laws (for example having a low-speed regime, a transitional regime and a high-speed regime) should not to be difficult. In this setting, the law operator in \eqref{eq:relation-u-p} takes the form
\begin{equation}
  \label{eq:law-jump}
  \bLambda(\bu) =
  \begin{cases}
    \bLambda_1(\bu) & \text{whenever $\norm{\bu} < \bar u$},\\
    \bLambda_2(\bu) & \text{whenever $\norm{\bu} > \bar u$}.
  \end{cases}
\end{equation}
Being able to impose a law at the interface $\{\norm{\bu} = \bar u\}$ is out of the scope of this paper---we will therefore consider our problem solved whenever we find a pressure field and a velocity field such that \eqref{eq:div-q}--\eqref{eq:relation-u-p} and \eqref{eq:law-jump} hold. This summarizes as finding functions $\bu$ and $p$ solving
\begin{equation}
  \label{eq:main-model}
  \begin{cases}
    \dive \bu = q & \text{on $\Omega$},\\
    \bLambda(\bu) = -\grad p + \bff & \text{on $\Omega$},\\
    \bu\cdot\bn = u_0  & \text{on $\Sigma_{\mt{v}}$},\\
    p = p_0  & \text{on $\Sigma_{\mt{p}}$},
  \end{cases}
\end{equation}
where $\bLambda$ is any law such as in \eqref{eq:law-jump}. Clearly, when $\bLambda_1\neq\bLambda_2$, this choice of adaptable law introduces a discontinuity at $\norm{\bu} = \bar u$ which we shall treat carefully when studying existence of solutions. We will complete in Section \ref{sec:assu-vel-press} the strong formulation \eqref{eq:main-model}, which is somewhat formal since it lacks an interface constitutive relation; indeed, we will introduce a multi-valued weak setting so as to be able to treat the interface without imposing any given law on it. We refer the reader to \cite{BMZ15} for a multi-valued monotone operator approach for pressure-dependent permeabilities; note that the operator therein is continuous in velocity.

\begin{rem}
  Whenever the boundary piece $\Sigma_{\mt{p}}$ is such that $\mt{Vol}^{d-1}(\Sigma_{\mt{p}}) = 0$, where $\mt{Vol}^{d-1}$ stands for the $(d-1)$-dimensional Lebesgue measure, it is classical to add to \eqref{eq:main-model} a constraint on the average of the pressure field:
\begin{equation}
  \label{eq:pressure-average}
  \frac{1}{\abs{\Omega}}\int_\Omega p = \varpi,
\end{equation}
for some $\varpi\in\R$. Indeed, this is often required to ensure uniqueness of the pressure field satisfying Problem \eqref{eq:main-model} when $\bLambda$ is a classical continuous law. Similarly we will impose \eqref{eq:pressure-average} whenever $\mt{Vol}^{d-1}(\Sigma_{\mt{p}}) = 0$, so that our problem in this case becomes: find functions $\bu$ and $p$ such that 
\begin{equation*}
  \begin{cases}
    \dive \bu = q & \text{on $\Omega$},\\
    \bLambda(\bu) = -\grad p + \bff & \text{on $\Omega$},\\
    \bu\cdot\bn = u_0  & \text{on $\p\Omega$},
  \end{cases}
\end{equation*}
under the constraint \eqref{eq:pressure-average}, where $\bLambda$ is now given in \eqref{eq:law-jump}. For ease of discussion, however, we will often silence \eqref{eq:pressure-average} and only refer to \eqref{eq:main-model} as being our problem, even when $\mt{Vol}^{d-1}(\Sigma_{\mt{p}}) = 0$; this average condition will nevertheless be naturally encoded in our weak formulation.
\end{rem}

We focus in this paper on Darcy-like operator laws, that is, laws involving no derivatives of the velocity field, so that we need to exclude Stokes' and Brinkman's equations as admissible examples. Under some additional conditions (depending in particular on the "direction" of the jump between laws $\bLambda_1$ and $\bLambda_2$ at the interface), we show existence of solutions via the study of an energetic formulation of \eqref{eq:main-model}. Indeed, we are able to define an energy functional on the space of velocity fields whose minimizers satisfy \eqref{eq:main-model}. Although we have uniqueness for this energetic formulation in some conditions (see Section \ref{sec:existence}), this property does not transfer to \eqref{eq:main-model}---this is natural since we are not imposing what the velocity-pressure law should be on the interface. In order to hope for uniqueness, one should either impose an adequate law on the interface or show that the interface must have zero Lebesgue measure so that it does not play a role in defining weak solutions. As already mentioned, we do not wish to focus on the problem of the interface at this stage and leave it to a later work.

\subsection{Intersecting fractures}
\label{subsec:intersecting-fractures}

In the case $\Omega$ is composed of multiple intersecting fracture branches, forming thus a
fracture network, we can extend the previous model by including suitable conditions at the intersections.
Given $\Omega$ we introduce $\omega^i \subset \Omega$ to be a fracture branch, with $n_\omega \ni i$ the 
total number of branches. Clearly, given two distinct branches $\omega^i$ and $\omega^j$,
with $i\neq j$, we have $\mathring{\omega}^i \cap \mathring{\omega}^j =
\emptyset$ and also that $\overline{\Omega} = \cup_{i =1}^{n_\omega}
\overline{\omega}^i$.

We consider $2 \leq n \leq n_\omega$ fracture branches that meet at an
intersection $\mathcal{I}$ whose closure $\overline{\mathcal{I}} = \cap_{i=1}^n \overline{\omega}^i$. To
complete model \eqref{eq:main-model} we impose the following conditions on $\mathcal{I}$:
\begin{gather}\label{eq:network_cc}
    \sum_{i=1}^n \bu^i \cdot \bn^i = 0
    \quad \text{and} \quad
    p^i = p^j, \quad \forall\, i,j=1, \ldots, n,
\end{gather}
where with a superscript $i$ we indicate the corresponding object restricted to
$\omega^i$, and $\bn^i$ is a unit vector tangent to $\omega^i$ and pointing to
$\mathcal{I}$, in the mono-dimensional case, and in addition normal to $\partial
\omega^i$, in the multi-dimensional case. This condition is frequently used, see
for instance \cite{Alboin2000a,Amir2005,Berrone2013a,Berrone2013}. The first condition
in \eqref{eq:network_cc} is a direct consequence of the conservation of mass at the
intersection, while the second can be derived from each constitutive relation of
the form \eqref{eq:relation-u-p}. These conditions do not put any additional
difficulties in the analysis and are therefore considered only in the numerical
examples.

It is possible to consider more complex conditions that allow pressure and
velocity jump at the intersection, see among others
\cite{Formaggia2012,Boon2018,Berre2020a}. However, since this is
not our main focus we keep the simpler interface condition \eqref{eq:network_cc}.

\section{Mathematical setting}
\label{sec:math-setting}

We give in this section the rigorous mathematical setting; in particular we introduce the assumptions on the underlying law operators as well as the weak formulation of our problem. We shall use the convention to use boldfaced symbols for vectors and vector-valued functions. 

From now on, without loss of generality we take $\bar u$ to be equal to $1$. For a given field $\bu\:\Omega\to \R^d$, we write 
\begin{gather}\label{eq:setsubdiv}
  \begin{gathered}
    \Omega_1(\bu) = \{ \bx\in \Omega \st \norm{\bu(\bx)} < 1 \}, \quad \Omega_2(\bu) = \{ \bx\in \Omega \st \norm{\bu(\bx)} > 1\},\\ \Gamma(\bu) = \Omega\setminus (\Omega_1(\bu) \cup \Omega_2(\bu)) = \{\bx\in \Omega \st \norm{\bu(\bx)} = 1 \},
  \end{gathered}
\end{gather}
where $\Omega_1(\bu)$, $\Omega_2(\bu)$ and $\Gamma(\bu)$ are what we have already respectively referred to as the \emph{low-speed region}, \emph{high-speed region} and \emph{interface} (associated with $\bu$). Obviously the family $\mathcal{C} := \{\Omega_1(\bu),\Omega_2(\bu),\Gamma(\bu)\}$ forms a partition of $\Omega$, and we will refer to $\mathcal{C}$ as the \emph{configuration} of the problem, especially for the numerics in Sections \ref{sec:numerical_approx} and \ref{sec:examples}. We can rewrite these sets in a more compact form:
\begin{equation*}
  \Omega_1(\bu) = \bu^{-1}(B_1(\bnull)), \quad \Omega_2(\bu) = \bu^{-1}(E_1(\bnull)), \quad \Gamma(\bu) = \bu^{-1}(S_1(\bnull)),
\end{equation*}
where $B_1(\bnull)$ and $S_1(\bnull)$ stand respectively for the unit open ball and unit sphere in $\R^d$ centered at the origin $\bnull$ and $E_1(\bnull) = \R^d\setminus (B_1(\bnull)\cup S_1(\bnull))$. Note that if $\bu$ is not continuous, then $\Omega_1(\bu)$ and $\Omega_2(\bu)$ may not be open.

For all $n\in[1,\infty)$, $m\in(0,\infty)$ and $A\subset \R^d$ measurable we will denote by $L^n(A)$ and $W^{m,n}(A)$ the Lebesgue space of measurable functions on $A$ with integrable $n$th power and the $m$th-order Sobolev space associated to $L^n(A)$; we will also write $\bL^n(A)$ in place of $(L^n(A))^d$. As usual, in these spaces, equality between two functions is always intended in the almost-everywhere sense.

\subsection{Assumptions on the velocity-pressure laws}
\label{sec:assu-vel-press}

In the following, the operator laws $\bLambda_1$ and $\bLambda_2$ are assumed to be of the form
\begin{equation}
  \label{eq:law-form}
  \bLambda_1(\bu) = \phi_1(\norm{\bu}^2)\bu \,\chi_{\Omega_1(\bu)} \quad \text{and} \quad \bLambda_2(\bu) = \phi_2(\norm{\bu}^2) \bu \,\chi_{\Omega_2(\bu)},
\end{equation}
where $\chi_A$ is the characteristic function of any set $A\subset\R^d$. The functions $\phi_1,\phi_2\: [0,\infty) \to [0,\infty)$ are continuous and increasing on $[0,1]$ and $[1,\infty)$, respectively. Furthermore, $\phi_2$ satisfies the following assumption: there exist $r\geq 2$ and $c,C>0$ such that
\begin{equation}
  \label{eq:bound-phi2}
   c a^{\frac{r-2}{2}} \leq \phi_2(a) \leq C \left(1+a^{\frac{r-2}{2}}\right) \quad \text{for all $a\geq 1$}.
\end{equation}
A recurrent notation we will use is
\begin{equation}
  \label{eq:lambdas}
  \lambda_1 := \phi_1(1) \quad \text{and} \quad \lambda_2 := \phi_2(1),
\end{equation}
and will call the difference $\lambda_2-\lambda_1$ the \emph{interface inverse permeability jump}. Symmetrically, whenever $k_1:=1/\lambda_1$ and $k_2:=1/\lambda_2$ are considered (cf. in particular Sections \ref{sec:numerical_approx} and \ref{sec:examples}), the difference $k_2 - k_1$ will be called the \emph{interface permeability jump}. We will see that the sign of this jump is an important threshold which determines the convexity of the energy functional underlying the problem. Note that because $\phi_2$ is increasing on $[1,\infty)$, one must have $c\leq\lambda_2$ and $C\geq \lambda_2/2$ in \eqref{eq:bound-phi2}. 

Interesting examples that fall into the above requests, in particular satisfying \eqref{eq:law-form} with \eqref{eq:bound-phi2}, are combinations of scalar versions of the Darcy and Darcy--Forchheimer laws $\bLambda_{\mt{D}}$ and $\bLambda_{\mt{DF}}$ (cf. \eqref{eq:F-ex}), as desired in the first place. Indeed, one is allowed to consider 
\begin{equation*}
    \bLambda_1(\bu) = \lambda_1 \bu \,\chi_{\Omega_1(\bu)} \quad \text{and} \quad \bLambda_2(\bu) = \lambda_2 \bu\,\chi_{\Omega_2(\bu)},
\end{equation*}
that is, $\phi_1 \equiv \lambda_1$ and $\phi_2 \equiv \lambda_2$, or
\begin{equation*}
    \bLambda_1(\bu) = \lambda_1 \bu \,\chi_{\Omega_1(\bu)} \quad \text{and} \quad \bLambda_2(\bu) = (\lambda_{21} + \lambda_{22} \norm{\bu}) \bu\,\chi_{\Omega_2(\bu)}.
\end{equation*}
that is, $\phi_1 \equiv \lambda_1$ and $\phi_2(a) = \lambda_{21} + \lambda_{22}\sqrt{a}$, where $\lambda_1$, $\lambda_2$, $\lambda_{21}$ and $\lambda_{22}$ are positive scalars. In the former case we would require $r=2$, whereas in the latter $r=3$.

\begin{rem}
\label{rem:tensor-perm}
  This setting where $\bLambda_1$ and $\bLambda_2$ are as in \eqref{eq:law-form} physically restricts us to scalar permeabilities. More general laws including tensor permeabilities, as motivated in \cite{SV01}, are for instance given by the following:
  \begin{equation*}
    \bLambda_1(\bu) = \phi_1(\blambda_1\bu\cdot\bu)\blambda_1\bu\, \chi_{\Omega_1(\bu)}\,
  \end{equation*}
and analogously for $\bLambda_2$, where $\blambda_1\in\R^{d\times d}$ is a symmetric positive definite matrix encoding a tensor permeability. Even more general forms are envisageable:
\begin{equation*}
    \bLambda_1(\bu) = \sum_{j=1}^n \phi_{1,j}(\blambda_{1,j}\bu\cdot\bu)\blambda_{1,j}\bu\, \chi_{\Omega_1(\bu)},
  \end{equation*}
where the $\phi_{1,j}$ and $\blambda_{1,j}$ are $n$ different law functions and permeability tensors and the dot $\cdot$ stands for the Euclidean inner product in $\R^d$. We leave these general laws to a future work. We claim that the techniques we use in the present paper for scalar laws should extend to tensor laws without too much difficulty.
\end{rem}

We denote by $s = \tfrac{r}{r-1}$ the conjugate exponent of $r$. Thanks to the continuity of $\phi_1$ and the right-hand inequality in \eqref{eq:bound-phi2}, we observe that the operators $\bLambda_1$ and $\bLambda_2$ map $\bL^r(\Omega)$ into $\bL^s(\Omega)$. Because we allow for any law on the interface, this invites us to consider the multi-valued setting where our combined law $\bLambda$ maps $\bL^r(\Omega)$ into the power set $2^{\bL^s(\Omega)}$ and is given by
\begin{equation*}
  \bLambda(\bu) =
  \begin{cases}
    \{\bLambda_1(\bu)\} & \text{on $\Omega_1(\bu)$},\\
    \bL^s(\Omega) & \text{on $\Gamma(\bu)$},\\
    \{\bLambda_2(\bu)\} & \text{on $\Omega_2(\bu)$},
  \end{cases}
\end{equation*}
which can equivalently be written as
\begin{equation}
    \label{eq:law-combined}
    \bLambda(\bu) = \{\bLambda_1(\bu) + \bLambda_2(\bu) + \bs{h}\chi_{\Gamma(\bu)} \}_{\bs{h}\in\bL^s(\Omega)}.
\end{equation}

\begin{rem}
  The existence results that we present in this paper still apply if we consider some ``background'' law with tensor permeability. Indeed, our proofs remain essentially untouched if the combined law $\bLambda$ in \eqref{eq:law-combined} is replaced by
  \begin{equation*}
    \bLambda(\bu) = \{\beta(\blambda\bu\cdot\bu)\blambda \bu + \bLambda_1(\bu) + \bLambda_2(\bu) + \bs{h}\chi_{\Gamma(\bu)} \}_{\bs{h}\in\bL^s(\Omega)},
  \end{equation*}
where $\blambda \in \R^{d\times d}$ is symmetric positive definite and $\beta\:[0,\infty) \to [0,\infty)$ is a continuous, increasing function such that $\beta + \phi_2$ satisfies \eqref{eq:bound-phi2} in place of $\phi_2$.
\end{rem}


\subsection{Weak formulation}
\label{sec:weak-formulation}

We fix $q\in L^r(\Omega)$, $u_0\in L^{r}(\Sigma_{\mt{v}})$, $p_0\in W^{\frac{1}{r},s}(\Sigma_{\mt{p}})$ and $\bff\in \bL^s(\Omega)$. Consider the Sobolev space
\begin{equation*}
  \widetilde W^{1,s}(\Omega) = 
  \begin{cases}
    \left\{ \xi \in W^{1,s}(\Omega) \st \displaystyle \frac{1}{\abs{\Omega}}\int_\Omega \xi = \varpi \right\} & \text{if $\mt{Vol}^{d-1}(\Sigma_{\mt{p}}) = 0$},\\
    \left\{ \xi \in W^{1,s}(\Omega) \st \xi = p_0\; \text{on $\Sigma_{\mt{p}}$} \right\} & \text{if $\mt{Vol}^{d-1}(\Sigma_{\mt{p}}) > 0$},
  \end{cases}
\end{equation*}
where we recall that $\mt{Vol}^{d-1}$ is the $(d-1)$-dimensional Lebesgue measure. The problem we shall focus on in the rest of the paper is: find $(\bu,p) \in \bL^r(\Omega)\times \widetilde W^{1,s}(\Omega)$ so that there exists $\bLambda_{\bu} \in \bLambda(\bu)$ such that for all $\bvarphi\in \bL^r(\Omega)$ and $\psi\in \widetilde W^{1,s}(\Omega)$ there holds
\begin{equation}
  \label{eq:weak-form-tilde}
  \begin{gathered}
    \int_\Omega \bLambda_{\bu} \cdot \bvarphi = - \int_\Omega \bgrad p \cdot\bvarphi + \int_\Omega \bff\cdot\bvarphi,\\
    \int_\Omega \bgrad \psi\cdot \bu = -\int_\Omega q\psi + \int_{\Sigma_{\mt{v}}} u_0\psi,
  \end{gathered}
\end{equation}
where we recall that $\bLambda$ is given in \eqref{eq:law-combined}. From the linearity of \eqref{eq:weak-form-tilde} with respect to the pressure field, we see that by setting $\bff_0 = \bff + \bgrad (E p_0)$, with $E\: W^{\frac{1}{r},s}(\Sigma_{\mt{p}}) \rightarrow W^{1,s}(\Omega)$ any extension operator being right-inverse of the $W^{1,s}(\Omega)$-trace operator, and defining the Sobolev space
\begin{equation*}
  W_0^{1,s}(\Omega) = 
  \begin{cases}
    \left\{ \xi \in W^{1,s}(\Omega) \st \displaystyle \int_\Omega \xi = 0 \right\} & \text{if $\mt{Vol}^{d-1}(\Sigma_{\mt{p}}) = 0$},\\
    \left\{ \xi \in W^{1,s}(\Omega) \st \xi = 0\; \text{on $\Sigma_{\mt{p}}$} \right\} & \text{if $\mt{Vol}^{d-1}(\Sigma_{\mt{p}}) > 0$},
  \end{cases}
\end{equation*}
the formulation in \eqref{eq:weak-form-tilde} is equivalent to: find $(\bu,p) \in \bL^r(\Omega)\times W_0^{1,s}(\Omega)$ so that there exists $\bLambda_{\bu} \in \bLambda(\bu)$ such that for all $\bvarphi\in \bL^r(\Omega)$ and $\psi\in W_0^{1,s}(\Omega)$ there holds
\begin{equation}
  \label{eq:weak-form}
  \begin{gathered}
    \int_\Omega \bLambda_{\bu} \cdot \bvarphi = - \int_\Omega \bgrad p \cdot\bvarphi + \int_\Omega \bff_0\cdot\bvarphi,\\
    \int_\Omega \bgrad \psi\cdot \bu = -\int_\Omega q\psi + \int_{\Sigma_{\mt{v}}} u_0\psi.
  \end{gathered}
\end{equation}
We endow $W_0^{1,s}(\Omega)$ with the norm $\norm{\psi}_{W_0^{1,s}(\Omega)} = \norm{\bgrad\psi}_{\bL^s(\Omega)}$ for all $\psi\in W_0^{1,s}(\Omega)$. We emphasize here that the well-posedness of \eqref{eq:weak-form} is not affected by the choice of the extension $E$ in the definition of $\bff_0$.

Thanks to Riesz's representation theorem, we shall equivalently manipulate functions in $\bL^s(\Omega)$ as elements of $(\bL^r(\Omega))^*$, the dual of $\bL^r(\Omega)$. In particular, this means that, given $\bu\in \bL^r(\Omega)$, operators $\bLambda_{\bu}$ in $\bLambda(\bu)$ will also be seen as maps from $\bL^r(\Omega)$ to $(\bL^r(\Omega))^* \sim \bL^s(\Omega)$. Writing $\ap{\cdot,\cdot}$ the dual mapping on $\bL^s(\Omega)\times \bL^r(\Omega)$, the weak formulation in \eqref{eq:weak-form} of our problem can be restated as follows: find $(\bu,p)\in \bL^r(\Omega)\times W_0^{1,s}(\Omega)$ so that there exists $\bLambda_{\bu} \in \bLambda(\bu)$ such that for all $\bvarphi\in \bL^r(\Omega)$ and $\psi\in W_0^{1,s}(\Omega)$ there holds
\begin{equation}
  \label{eq:weak-form-main}
  \begin{gathered}
    \ap{\bLambda_{\bu},\bvarphi} = -\ap{\bgrad p,\bvarphi} + \ap{\bff_0,\bvarphi},\\
    \ap{\bgrad\psi,\bu} = - \int_\Omega q\psi + \int_{\Sigma_{\mt{v}}} u_0\psi .
  \end{gathered}
\end{equation}
Our goal is now to show existence for this weak formulation. 

\begin{rem}
\label{rem:uniqueness-interface}
    We repeat here that, given the form of the law $\bLambda$ in \eqref{eq:law-form}, we cannot get uniqueness of solutions---in fact of velocity fields---satisfying \eqref{eq:weak-form-main} since it is embedded in the very formulation of the problem that any interface velocity-pressure law is admissible. Uniqueness could be obtained if an appropriate interface law were imposed or if the interface were shown to be less than $d$-dimensional and therefore had no contribution in the definition of weak solutions.
\end{rem}

\section{Existence}
\label{sec:existence}

We state and prove here our results on existence for Problem \eqref{eq:weak-form-main}. As already mentioned, we will see that our results depend on the sign of the interface inverse permeability jump, $\lambda_2 - \lambda_1$. The strategy is the following: 
\begin{enumerate}
  \item we reduce Problem \eqref{eq:weak-form} on the velocity and pressure fields into an equivalent problem on the velocity field only (cf. \eqref{eq:weak-form-V0} and \cite{AFM18,SLM13});
  \item we derive an energetic formulation whose minimizers are solutions to this reduced problem on the velocity field (cf. \eqref{eq:minimization} and \cite{SV01});
  \item we study the existence of minimizers for this energetic formulation distinguishing the convex case $\lambda_1\leq\lambda_2$ from the non-convex case $\lambda_1 > \lambda_2$ (cf. Theorems \ref{thm:wp-1-less-than-2} and \ref{thm:wp-2-less-than-1}). We are only able to treat the one-dimensional setting $d=1$ when $\lambda_1 > \lambda_2$. 
\end{enumerate}

\subsection{Reduction to a problem on the velocity field}
\label{sec:restriction-V0}

Define the maps $B \: \bL^r(\Omega) \to (W_0^{1,s}(\Omega))^*$ and $\bB^* \: W_0^{1,s}(\Omega) \to (\bL^r(\Omega))^*$  by
\begin{equation}
  \label{eq:B}
  B(\bvarphi)(\psi) = \bB^*(\psi)(\bvarphi) = \ap{\bgrad \psi, \bvarphi} \quad \text{for all $\bvarphi \in \bL^r(\Omega)$ and $\psi\in W_0^{1,s}(\Omega)$},
\end{equation}
Let us introduce the set
\begin{equation*}
  V = \{ \bvarphi \in \bL^r(\Omega) \st \forall\, \psi\in W_0^{1,s}(\Omega),\, \ap{\bgrad \psi,\bvarphi} = 0 \},
\end{equation*}
which satisfies $V = \mathrm{Ker}(B)$. We write $V^\perp\subset \bL^s(\Omega)$ the polar subspace of $V$, that is, $V^\perp = \{ \bg \in \bL^s(\Omega) \st \forall\, \bvarphi\in V,\, \ap{\bg,\bvarphi} = 0\}$, and $\bL^r(\Omega)/V \subset \bL^r(\Omega)$ the quotient space of $\bL^r(\Omega)$ by $V$.

The following three lemmas are inspired from their equivalents in \cite{AFM18}.

\begin{lem}
  \label{lem:isomorphisms}
  The maps $B$ and $B^*$ in \eqref{eq:B} are isomorphisms from $\bL^r(\Omega)/V$ to $(W_0^{1,s}(\Omega))^*$ and from $W_0^{1,s}(\Omega)$ to $V^\perp$, respectively.
\end{lem}
\begin{proof}
  By \cite[Lemma 2.1]{AFM18}, it suffices to show that there exists $\gamma > 0$ such that
    \begin{equation*}
    \inf_{\substack{\psi \in W_0^{1,s}(\Omega)\\ \psi\neq 0}} \sup_{\substack{\bvarphi\in \bL^r(\Omega)\\ \bvarphi\neq \bnull}} \frac{\ap{\bgrad\psi,\bvarphi}}{\norm{\psi}_{W_0^{1,s}(\Omega)}\norm{\bvarphi}_{\bL^r(\Omega)}} \geq \gamma.
  \end{equation*}
  To this end, let $\psi\in W_0^{1,s}(\Omega)$ with $\psi \neq 0$. Then, by identifying $\bL^s(\Omega)$ with $(\bL^r(\Omega))^*$ we get
  \begin{equation*}
    \norm{\psi}_{W_0^{1,s}(\Omega)} = \norm{\bgrad\psi}_{\bL^s(\Omega)} = \sup_{\substack{\bvarphi\in \bL^r(\Omega)\\\bvarphi\neq\bnull}} \frac{\ap{\bgrad\psi,\bvarphi}}{\norm{\bvarphi}_{\bL^r(\Omega)}},
  \end{equation*}
  so that
  \begin{equation*}
    1 = \frac{\norm{\psi}_{W_0^{1,s}(\Omega)}}{\norm{\psi}_{W_0^{1,s}(\Omega)}} = \sup_{\substack{\bvarphi\in \bL^r(\Omega)\\\bvarphi\neq\bnull}} \frac{\ap{\bgrad\psi,\bvarphi}}{\norm{\psi}_{W_0^{1,s}(\Omega)}\norm{\bvarphi}_{\bL^r(\Omega)}}.
  \end{equation*}
  Taking above the infimum over all $\psi \in W_0^{1,s}(\Omega)$ with $\psi\neq 0$ ends the proof.
\end{proof}

Now, via the following two lemmas, we simplify our weak formulation \eqref{eq:weak-form-main} into a problem restricted to $V$ (cf. \eqref{eq:weak-form-V0} below). 
\begin{lem}
  \label{lem:bdry-V0}
  There exists a unique $\hat\bu\in \bL^r(\Omega)/V$ such that
  \begin{equation*}
    \ap{\bgrad \psi,\hat\bu} = -\int_\Omega q\psi + \int_{\Sigma_{\mt{v}}} u_0\psi \quad \text{for all $\psi\in W_0^{1,s}(\Omega)$}.
  \end{equation*}
\end{lem}
\begin{proof}
  For all $\psi\in W_0^{1,s}(\Omega)$, let
  \begin{equation*}
    F(\psi) = -\int_\Omega q\psi + \int_{\Sigma_{\mt{v}}} u_0\psi. 
  \end{equation*}
  Since we are assuming that $q\in L^r(\Omega)$ and $u_0\in L^r(\Sigma_{\mt{v}})$, the map $F$ is a well-defined linear map from $W_0^{1,s}(\Omega)$ into $\R$, i.e., $F\in (W_0^{1,s}(\Omega))^*$. By Lemma \ref{lem:isomorphisms}, we thus know there exists a unique $\hat\bu \in \bL^r(\Omega)/V$ such that
  \begin{equation*}
    \ap{\bgrad\psi,\hat\bu} = B(\hat\bu)(\psi) = F(\psi) = -\int_\Omega q\psi + \int_{\Sigma_{\mt{v}}} u_0\psi \quad \text{for all $\psi\in W_0^{1,s}(\Omega)$}, 
  \end{equation*}
  which is the desired result.
\end{proof}
\begin{lem}
  The weak formulation \eqref{eq:weak-form-main} is equivalent to the following problem: find $\bv\in V$ such that there exists $\bLambda_{\bv} \in \bLambda(\bv+\hat\bu)$ satisfying
  \begin{equation}
    \label{eq:weak-form-V0}
    \ap{\bLambda_{\bv},\bvarphi} = \ap{\bff_0,\bvarphi} \quad \text{for all $\bvarphi \in V$},
  \end{equation}
  where $\hat\bu$ is given by Lemma \ref{lem:bdry-V0}.
\end{lem}
\begin{proof}
  We first suppose that $(\bu,p)$ is solution to Problem \eqref{eq:weak-form-main}. Then we decompose $\bu$ as $\bu = (\bu - \hat\bu) + \hat\bu =: \bv + \hat\bu$. By \eqref{eq:weak-form-main}, there exists $\bLambda_{\bv} \in \bLambda(\bu) = \bLambda(\bv+\hat\bu)$ such that
  \begin{equation*}
    \ap{\bLambda_{\bv},\bvarphi} = \ap{\bff_0,\bvarphi} \quad \text{for all $\bvarphi \in V$}.
  \end{equation*}
Furthermore, for all $\psi\in W_0^{1,s}(\Omega)$ we find $\ap{\bgrad\psi,\bv} = \ap{\bgrad\psi,\bu} - \ap{\bgrad\psi,\hat\bu} = 0$, so that $\bv\in V$ and $\bv$ satisfies Problem \eqref{eq:weak-form-V0}.

  Suppose now that $\bv\in V$ satisfies Problem \eqref{eq:weak-form-V0}. Then, $\bLambda_{\bv} - \bff_0 \in V^{\mathrm{\perp}}$. By Lemma \ref{lem:isomorphisms} we know $\bB^*$ is an isomorphism from $W_0^{1,s}(\Omega)$ to $V^{\mathrm{\perp}}$, so that there exists a unique $p\in W_0^{1,s}(\Omega)$ with
  \begin{equation*}
    \ap{\bLambda_{\bv} - \bff_0,\bvarphi} = \bB^*(-p)(\bvarphi) = -\ap{\bgrad p,\bvarphi} \quad \text{for all $\bvarphi\in\bL^r(\Omega)$}.
  \end{equation*}
Furthermore, writing $\bu = \bv + \hat\bu$ we get
\begin{equation*}
  \ap{\bgrad\psi,\bu} = \ap{\bgrad\psi,\bv} + \ap{\bgrad\psi,\hat\bu} = \ap{\bgrad\psi,\hat\bu}  = - \int_\Omega q\psi + \int_{\Sigma_{\mt{v}}} u_0\psi,
\end{equation*}
since $\bv\in V$. We thus have that $(\bu,p)$ satisfies Problem \eqref{eq:weak-form-main}, with $p$ unique.
\end{proof}

\subsection{Energetic formulation}
\label{sec:energ-form}


We first give the definition of Fréchet subdifferential and strong local minimizer in our setting:
\begin{defn}[Fréchet subdifferential and strong local minimizer]
  Let $\F\:V\to \R$. For all $\bv\in V$ we define the \emph{(Fréchet) subdifferential} $\p\F(\bv)$ of $\F$ at $\bv$ by:
  \begin{equation*}
    \bg_{\bv} \in \p\F(\bv) \iff \bg_{\bv} \in V^*\; \text{ and }\; \forall\, \bvarphi\in V,\; \liminf_{\delta\to 0^+} \frac{\F(\bv+\delta\bvarphi) - \F(\bv)}{\delta} \geq \ap{\bg_{\bv},\bvarphi}.
  \end{equation*}
  We say that $\bv\in V$ is a \emph{(strong) local minimizer} of $\F$ if there exists $\eta>0$ such that for all $\bvarphi\in V$ we have $\F(\bv +\delta\bvarphi) \geq \F(\bv)$ for all $\delta\in[0,\eta)$.
\end{defn}

\begin{rem}
  \label{rem:minimizers}
  The important property of the subdifferential to keep in mind here is that if $\bv\in V$ is a local minimizer of a functional $\F\:V\to \R$, then $\bnull_{\bL^s(\Omega)}\in \p\F(\bv)$. When $\F$ is convex, the reverse of this statement is also true: if $\bnull_{\bL^s(\Omega)}\in \p\F(\bv)$, then $\bV$ is a local minimizer of $\F$.
\end{rem}

We write $\Psi\:[0,\infty)\to \R$ the function given by
\begin{equation}
  \label{eq:psi}
  \Psi(a) = \begin{cases} \Phi_1(a) & \text{for all $a\leq 1$},\\ \Phi_2(a) & \text{for all $a> 1$}, \end{cases}
\end{equation}
where $\Phi_1,\Phi_2\:[0,\infty) \to \R$ are primitives of $\tfrac{\phi_1}{2}$ and $\tfrac{\phi_2}{2}$ (cf. \eqref{eq:law-form}) such that $\Phi_1(1) = \Phi_2(1) = 0$. We define the \emph{dissipation} $\D\:V \to \R$ by
\begin{equation}
  \label{eq:dissipation}
  \D(\bv) = \int_\Omega \Psi(\norm{\bv + \hat\bu}^2) \quad \text{for all $\bv\in V$}.
\end{equation}
Thanks to our assumptions on $\phi_1$ and $\phi_2$ we easily get that the domain of $\D$ is indeed all of $V$, and it is thus a well-defined functional from $V$ into $\R$. Consider the following problem: find $\bv\in V$ such that there exists $\bg_{\bv}\in \p\D(\bv)$ satisfying
\begin{equation}
  \label{eq:energy-form-V0}
  \ap{\bg_{\bv},\bvarphi} = \ap{\bff_0,\bvarphi} \quad \text{for all $\bvarphi\in V$}.
\end{equation}
The following result holds:
\begin{lem}
  \label{lem:energy-implies-sol}
  Any solution to Problem \eqref{eq:energy-form-V0} is also solution to Problem \eqref{eq:weak-form-V0}.
\end{lem}
\begin{proof}
  Suppose that $\bv\in V$ is solution to Problem \eqref{eq:energy-form-V0}. Then we can pick $\bg_{\bv}\in\p\D(\bv)\subset\bL^s(\Omega)$ so that \eqref{eq:energy-form-V0} holds. By definition, for all $\bvarphi\in V$ we must have
  \begin{equation*}
    \liminf_{\delta\to0^+} \frac{\D(\bv+\delta\bvarphi) - \D(\bv)}{\delta} \geq \ap{\bg_{\bv},\bvarphi}.
  \end{equation*}
    Write $V_1$ the subset of $V$ consisting of the functions which are supported in $\Omega_1(\bv+\hat\bu)$. Then, for all $\bvarphi\in V_1$, since $-\bvarphi\in V_1$ as well, we get
  \begin{equation*}
    \liminf_{\delta\to0^+} \frac{1}{\delta} \int_{\Omega_1(\bv+\hat\bu)} \left( \Psi(\norm{\bv +\delta\bvarphi + \hat\bu}^2) - \Phi_1(\norm{\bv + \hat\bu}^2) \right) \geq \ap{\bg_{\bv},\bvarphi}
  \end{equation*}
  and
  \begin{equation*}
    \liminf_{\delta\to0^+} \frac{1}{\delta} \int_{\Omega_1(\bv+\hat\bu)} \left( \Psi(\norm{\bv -\delta\bvarphi + \hat\bu}^2) - \Phi_1(\norm{\bv + \hat\bu}^2) \right) \geq -\ap{\bg_{\bv},\bvarphi},
  \end{equation*}
  which, by Lebesgue's dominated convergence theorem and the differentiability of $\Phi_1$, yield
  \begin{equation*}
    \int_\Omega \bLambda_1(\bv+\hat\bu)\cdot\bvarphi = \int_{\Omega_1(\bv+\hat\bu)} \phi_1(\norm{\bv+\hat\bu}^2) (\bv +\hat\bu)\cdot\bvarphi  \geq \ap{\bg_{\bv},\bvarphi}
  \end{equation*}
  and
  \begin{equation*}
    -\int_\Omega \bLambda_1(\bv+\hat\bu)\cdot\bvarphi = -\int_{\Omega_1(\bv+\hat\bu)} \phi_1(\norm{\bv+\hat\bu}^2) (\bv +\hat\bu)\cdot\bvarphi  \geq -\ap{\bg_{\bv},\bvarphi}.
  \end{equation*}
  All in all we get
  \begin{equation*}
    \ap{\bg_{\bv},\bvarphi} = \ap{\bLambda_1(\bv+\hat\bu),\bvarphi} \quad \text{for all $\bvarphi\in V_1$}.
  \end{equation*}
  Similarly, denoting by $V_2$ the subset of $V$ consisting of the functions which are supported in $\Omega_2(\bv+\hat\bu)$, we get
  \begin{equation*}
    \ap{\bg_{\bv},\bvarphi} = \ap{\bLambda_2(\bv+\hat\bu),\bvarphi} \quad \text{for all $\bvarphi\in V_2$}.
  \end{equation*}
  Thus the function defined by
  \begin{equation*}
    \bLambda_{\bv} = \bLambda_1(\bv+\hat\bu) + \bLambda_2(\bv+\hat\bu) + \bg_{\bv} \,\chi_{\Gamma(\bv+\hat\bu)}
  \end{equation*}
  is such that $\ap{\bLambda_{\bv},\bvarphi} = \ap{\bg_{\bv},\bvarphi}$ for all $\bvarphi\in V$. Therefore, from \eqref{eq:energy-form-V0} we obtain
  \begin{equation*}
    \ap{\bLambda_{\bv},\bvarphi} = \ap{\bff_0,\bvarphi} \quad \text{for all $\bvarphi\in V$}.
  \end{equation*}
  Moreover $\bLambda_{\bv}\in\bLambda(\bv+\hat\bu)$, where we recall $\bLambda$ is in \eqref{eq:law-combined}. Hence $\bv$ is solution to Problem \eqref{eq:weak-form-V0}.
\end{proof}

We now define the \emph{energy} $\E\:V \to \R$ associated to $\D$ by 
\begin{equation}
  \label{eq:energy}
  \E(\bv) = \D(\bv) - \ap{\bff_0,\bv + \hat\bu} \quad \text{for all $\bv\in V$}.
\end{equation}
Let us write $M_\E\subset V$ the, possibly empty, set of local minimizers of $\E$. Consider the following minimization problem: find $\bv\in V$ such that 
\begin{equation}
  \label{eq:minimization}
  \bv \in M_\E.
\end{equation}
By Remark \ref{rem:minimizers}, any solution to Problem \eqref{eq:minimization} is also solution to Problem \eqref{eq:energy-form-V0}; if $\E$ is convex these problems are actually equivalent. By Lemma \ref{lem:energy-implies-sol} it therefore suffices to find a solution to Problem \eqref{eq:minimization} in order to get the desired existence result on our original problem given in \eqref{eq:weak-form-main}. The rest of this section will thus be solely dedicated to solving Problem \eqref{eq:minimization}.

\subsection{Case $\lambda_1\leq\lambda_2$}
\label{sec:case-1-less-than-2}

The result we wish to show here is the following:
\begin{thm}[Existence and uniqueness when $\lambda_1 \leq \lambda_2$]
  \label{thm:wp-1-less-than-2}
  Suppose that $\lambda_1 \leq \lambda_2$. Then, Problem \eqref{eq:minimization} has a unique solution.
\end{thm}
\begin{proof}
  Let us prove that the integrand $\Psi\circ\norm{\cdot}^2$ (cf. \eqref{eq:psi}) of $\D$ is strictly convex. Define the functions $\bar\Phi_1,\bar\Phi_2\:[0,\infty)\to\R$ by
  \begin{equation*}
    \bar \Phi_1(a) = 
    \begin{cases}
      \Phi_1(a) & \text{for all $a\leq1$},\\
      \tfrac{\lambda_1}{2}(a-1) & \text{for all $a> 1$},
    \end{cases}
    \quad \text{and} \quad 
    \bar \Phi_2(a) = 
    \begin{cases}
      \tfrac{\lambda_2}{2}(a-1) & \text{for all $a\leq1$},\\
      \Phi_2(a) & \text{for all $a> 1$}.\\
    \end{cases}
  \end{equation*}
  Then $\bar\Phi_1$ and $\bar\Phi_2$ are differentiable with 
  \begin{equation*}
    \bar \Phi_1'(a) = \frac12
    \begin{cases}
      \phi_1(a) & \text{for all $a\leq1$},\\
      \lambda_1 & \text{for all $a> 1$},
    \end{cases}
    \quad \text{and} \quad 
    \bar \Phi_2'(a) = \frac12
    \begin{cases}
      \lambda_2 & \text{for all $a\leq1$},\\
      \phi_2(a) & \text{for all $a> 1$}.\\
    \end{cases}
  \end{equation*}
  Since $\phi_1$ and $\phi_2$ are increasing on $[0,1]$ and $[1,\infty)$, respectively, and $\phi_1(1) = \lambda_1$ and $\phi_2(1) = \lambda_2$, the derivatives $\bar\Phi_1'$ and $\bar\Phi_2'$ are also increasing so that $\bar\Phi_1$ and $\bar\Phi_2$ are convex. Furthermore, because $\lambda_1\leq \lambda_2$ and $\phi_2$ is increasing, we have $\bar\Phi_2'(a) \geq \bar\Phi_1'(a)$ for all $a>1$; thus, the fact that $\Phi_1(1) = \Phi_2(1)$ (and so $\bar \Phi_1(1) = \bar \Phi_2(1)$) yields $\bar\Phi_1(a)\leq \bar\Phi_2(a)$ for all $a>1$. Similarly, we get that $\bar\Phi_1(a)\geq \bar\Phi_2(a)$ for all $a<1$. Let $a,b \in [0,\infty)$ and $t\in[0,1]$. If $a,b \leq 1$, then $(1-t)a+tb\leq1$ and
  \begin{align*}
    \Psi((1-t) a + tb) &= \Phi_1((1-t) a + tb) = \bar\Phi_1((1-t) a + tb)\\
    &\leq (1-t)\bar\Phi_1(a) + t\bar\Phi_1(b) = (1-t)\Psi(a) + t\Psi(b),
  \end{align*}
  and similarly if $a,b > 1$. If now $a\leq1$, $b>1$ and $(1-t)a+tb\leq1$, then we have
  \begin{align*}
    \Psi((1-t) a + tb) &= \Phi_1((1-t) a + tb) = \bar\Phi_1((1-t) a + tb)\\
    &\leq (1-t)\bar\Phi_1(a) + t\bar\Phi_1(b) \leq (1-t)\bar\Phi_1(a) + t\bar\Phi_2(b) = (1-t)\Psi(a) + t\Psi(b),
  \end{align*}
  and similarly if $(1-t)a+tb>1$ or $a>1$ and $b\leq1$. In all cases, we see that
  \begin{equation*}
    \Psi((1-t)a+tb) \leq (1-t)\Psi(a) + t\Psi(b),
  \end{equation*}
  so that $\Psi$ is convex. Note that one could reach the same conclusion using that 
  \begin{equation*}
      \Psi(a) = \max(\bar \Phi_1(a),\bar \Phi_2(a)) \quad \text{for all $a\geq0$}.
  \end{equation*}
  Since $\phi_1$ and $\phi_2$ are positive on $(0,1)$ and $[1,\infty)$, respectively, we get that $\Psi$ is increasing. Therefore, the function $\Psi\circ\norm{\cdot}^2$ is strictly convex.  

  We now want to use the direct method of the calculus of variations to show that $\E$ has in fact a unique global (and thus local) minimizer. Let $(\bv_n)_{n\in\N} \subset V$ be a minimizing sequence for $\E$. Then we know there exists $N\in\N$ large enough and $K > 0$ such that $\E(\bv_n) < K$ for all $n>N$. Without loss of generality we can therefore assume that the sequence $(\E(\bv_n))_{n\in\N}$ is bounded by some constant $K>0$. Hence, thanks to the left-hand inequality in \eqref{eq:bound-phi2}, for all $n\in\N$ we have
  \begin{align*}
    K > \E(\bv_n) &= \int_{\Omega_1(\bv_n+\hat\bu)} \Phi_1(\norm{\bv_n+\hat\bu}^2) + \int_{\Omega_2(\bv_n+\hat\bu)} \Phi_2(\norm{\bv_n+\hat\bu}^2)- \int_\Omega \bff_0\cdot(\bv_n+\hat\bu)\\
              &\geq \Phi_1(0) \abs{\Omega} + \frac{c}{r}\int_{\Omega_2(\bv_n+\hat\bu)} \norm{\bv_n+\hat\bu}^r - \norm{\bff_0}_{\bL^s(\Omega)}^s\norm{\bv_n+\hat\bu}_{\bL^r(\Omega)}^r\\
              &\geq \left(\Phi_1(0)-\frac{c}{r}\right) \abs{\Omega} + \frac{c}{r}\norm{\bv_n+\hat\bu}_{\bL^r(\Omega)}^r - \norm{\bff_0}_{\bL^s(\Omega)}\norm{\bv_n+\hat\bu}_{\bL^r(\Omega)},
  \end{align*}
which shows that the sequence $(\norm{\bv_n}_{\bL^r(\Omega)})_{n\in\N}$ is bounded. Thus we can extract a subsequence from $(\bv_n)_{n\in\N}$, still denoted $(\bv_n)_{n\in\N}$ which converges weakly to some $\bv\in\bL^r(\Omega)$. Since further $V$ is weakly closed we in fact have $\bv\in V$. Because $\Psi\circ\norm{\cdot}^2$ is convex, the dissipation $\D$ is weakly lower semi-continuous and so is the energy $\E$. We therefore yield
\begin{equation*}
  \inf_{\bw\in V}\E(\bw) = \liminf_{n\to\infty} \E(\bv_n) \geq \E(\bv) \geq \inf_{\bw\in V}\E(\bw),
\end{equation*}
so that $\E(\bv) = \inf_{\bw\in V}\E(\bw)$ and $\bv$ is a global minimizer of $\E$ and so $\bv\in M_\E$.

To show that $M_\E$ is a singleton it is enough to prove that $\E$ is strictly convex. Let $\bv,\bw\in V$ and $t\in(0,1)$. Suppose furthermore that $\bv\neq\bw$ and write $A\subset\Omega$ the set where $\bv$ and $\bw$ are different; the Lebesgue measure of $A$ is therefore positive. Note that we must have $\Psi(a)\neq0$ for all $a\neq1$. Using the strict convexity of $\Psi\circ\norm{\cdot}^2$ we therefore get
  \begin{align*}
    \D((1-t)\bv+t\bw) &= \int_\Omega\Psi(\norm{(1-t)\bv+t\bw+\hat\bu}^2) = \int_\Omega \Psi(\norm{(1-t)(\bv+\hat\bu)+t(\bw+\hat\bu)}^2)\\
                      &= (1-t)\int_{\Omega\setminus A} \Psi(\norm{\bv+\hat\bu}^2) + t\int_{\Omega\setminus A} \Psi(\norm{\bw+\hat\bu}^2)\\
                      &\quad + \int_{A} \Psi(\norm{(1-t)(\bv+\hat\bu)+t(\bw+\hat\bu)}^2)\\
                      &< (1-t)\int_{\Omega\setminus A} \Psi(\norm{\bv+\hat\bu}^2) + t\int_{\Omega\setminus A} \Psi(\norm{\bw+\hat\bu}^2)\\
                      &\quad + (1-t)\int_{A} \Psi(\norm{\bv+\hat\bu}^2) +t\int_{A} \Psi(\norm{\bw+\hat\bu}^2)\\
                      &= (1-t)\int_\Omega \Psi(\norm{\bv+\hat\bu}^2) +t\int_\Omega \Psi(\norm{\bw+\hat\bu}^2) = (1-t)\D(\bv) + t\D(\bw),
  \end{align*}
  so that the dissipation $\D$ is strictly convex. Consequently, the energy $\E$ is also strictly convex and the proof is over. 
\end{proof}

\subsection{Case $\lambda_1 > \lambda_2$}
\label{sec:case-2-less-than-1}

This case is more difficult to tackle than the case $\lambda_1 \leq \lambda_2$. Indeed, we lose the convexity of the integrand $\Psi\circ\norm{\cdot}^2$ (cf. proof of Theorem \ref{thm:wp-1-less-than-2}), which by Tonelli's theorem of functional analysis means that $\E$ is \emph{not} weakly lower semi-continuous. As a consequence we cannot use the direct method of the calculus of variations in the weak topology to deduce the existence of a minimizer of $\E$. 

To simplify our task at this point, we shall restrict to the one-dimensional case (i.e., $d=1$) and leave the higher-dimensional case for future investigation. The results we present here therefore apply to the numerical experiments we present below. The one-dimensional case is simpler since we can easily characterize the space $V$ depending on the boundary conditions, as will be clear from the proof of the following theorem:
\begin{thm}[Existence and uniqueness when $\lambda_1 > \lambda_2$]
  \label{thm:wp-2-less-than-1}
  Let $d=1$ and suppose that $\lambda_1 > \lambda_2$. If $\mt{Vol}^0(\Sigma_{\mt{v}}) = 0$, then Problem \eqref{eq:minimization} has a solution. If instead $\mt{Vol}^0(\Sigma_{\mt{v}}) > 0$, then Problem \eqref{eq:minimization} has a unique solution.
\end{thm}
\begin{proof}
Without loss of generality, take $\Omega = (0,1)$. 

\medskip
\ul{Case $\mt{Vol}^0(\Sigma_{\mt{v}}) = 0$.} First note that here $\Sigma_{\mt{p}} = \{0,1\}$. Let us characterize $V$ in this case. To this end, note that any $\eta\in V$ with $\int_0^1 \eta = 0$ has a primitive in $W_0^{1,s}((0,1))$; indeed, the function $\psi\: (0,1) \to \R$ defined by
\begin{equation*}
    \psi(x) = \int_0^x \eta \quad \text{for all $x\in(0,1)$}  
\end{equation*}
satisfies $\psi' = \eta$ and $\psi(0) = \psi(1) = 0$. Let now $v \in V$ and define $\eta\in V$ as
\begin{equation*}
    \eta(x) = v(x) - \int_0^1 v  \quad \text{for all $x\in(0,1)$},
\end{equation*}
so that obviously $\int_0^1 \eta = 0$. Write $\psi\in W_0^{1,s}((0,1))$ a primitive of $\eta$ and compute, for all $\varphi\in V$,
\begin{equation*}
    \int_0^1 \left( v - \int_0^1 v \right) \varphi = \int_0^1 \eta\varphi = \int_0^1 \psi' \varphi = 0.
\end{equation*}
Thus $v = \int_0^1 v$ and $v$ is constant. Since constant functions clearly belong to $V$, this shows that $V$ is in fact the set of all constant functions on $(0,1)$ and we can identify $V$ with $\R$. 

This in particular means that the energy $\E$ defined in \eqref{eq:energy} can be identified with the following function $E\:\R\to\R$:
\begin{equation*}
   E(\alpha) = \int_0^1 \Psi((\alpha+\hat u)^2) - \alpha \int_0^1 f_0 \quad \text{for all $\alpha\in\R$},
\end{equation*}
where we recall that $\Psi$ is given in \eqref{eq:psi} and $\hat u := \hat\bu$ is as in Lemma \ref{lem:bdry-V0}. We can use the direct method of the calculus of variations on $E$ in the Euclidean topology in $\R$. Let $(\alpha_n)_{n\in\N} \subset \R$ be a minimizing sequence for $E$. Using a similar calculation as in the proof of Theorem \ref{thm:wp-1-less-than-2}, thanks to \eqref{eq:bound-phi2} we can show that $(\alpha_n)_{n\in\N}$ is bounded in $\R$. Therefore, there exists a subsequence of $(\alpha_n)_{n\in\N}$, still denoted by $(\alpha_n)_{n\in\N}$,  converging to some $\alpha\in\R$. By continuity of $\Psi$ and Fatou's lemma we get that $E$ is lower semi-continuous on $\R$, so that
\begin{equation*}
  \inf_{\beta\in\R} E(\beta) = \liminf_{n\to\infty} E(\alpha_n) \geq E(\alpha) \geq \inf_{\beta\in\R} E(\beta),
\end{equation*}
showing that $\alpha$ is a global minimizer of $E$ and the constant function $\alpha$ belongs to $M_\E$.

\medskip
\ul{Case $\mt{Vol}^0(\Sigma_{\mt{v}}) > 0$.} Note that here $\Sigma_{\mt{p}} \in \{\emptyset\} \cup \{\{0\}\} \cup \{\{1\}\}$. Similarly to the previous case, let us characterize $V$. Note that any $v\in V$ has a primitive in $W_0^{1,s}((0,1))$; indeed, the function $\psi\: (0,1) \to \R$ defined for all $x\in(0,1)$ by
\begin{equation*}
  \begin{cases}
    \psi(x) = \displaystyle \int_0^x v - \int_0^1 \left(\int_0^y v\right) \d y & \text{if $\Sigma_{\mt{p}}=\emptyset$},\\[10pt]
    \psi(x) = \displaystyle \int_0^x v - \int_0^1 v & \text{if $\Sigma_{\mt{p}} = \{1\}$},\\[10pt]
    \psi(x) = \displaystyle \int_0^x v & \text{if $\Sigma_{\mt{p}} = \{0\}$},\\
  \end{cases}
\end{equation*}
satisfies $\psi' = \eta$, and $\int_0^1\psi = 0$ if $\Sigma_{\mt{p}}=\emptyset$ and $\psi(x) = 0$ for $x\in\Sigma_{\mt{p}}$ otherwise. Let $v\in V$, write $\psi$ a primitive of $v$ and compute, for all $\varphi\in V$,
\begin{equation*}
  \int_0^1 v\varphi = \int_0^1 \psi' \varphi = 0,
\end{equation*}
so that $v=0$. This shows that $V = \{0\}$, i.e., $V$ contains only the zero function on $(0,1)$. Trivially then, the zero function is the unique global minimizer of the energy $\E$ in \eqref{eq:energy} and $M_\E$ is a singleton. Note that in this case the energy is only trivially convex.
\end{proof}

\section{Numerical approximation}\label{sec:numerical_approx}

In this part we present the numerical approximation for the considered problem when $d=1$; we will
still keep boldfaced notation for vectors and vector-valued functions for coherence with most of
the previous sections. In particular, in Section \ref{subsec:interface_algorithm} the algorithm to track the interface is
described. Later in Section \ref{subsec:discretization_space} we describe the
discretization adopted for a known interface. As mentioned in the introduction, for the implementation we have used
the flexible framework of the PorePy library; see \cite{Keilegavlen2020}.

We introduce a mesh $\Omega_h$ composed of non-overlapping segments
$E\in\Omega_h$ that approximate $\Omega$; we clearly have $\overline{\Omega}_h
= \cup_{E\in\Omega_h} \overline{E}$.  At the discrete level we can define the approximate
configuration $\mathcal{C}_h$ which is given by the set $\mathcal{C}_h =
\{\Omega_{1, h}, \Omega_{2, h}, \Gamma_h\}$, where respectively $\Omega_{1, h}
\subset \Omega_h$ and $\Omega_{2, h} \subset \Omega_h$ are the approximations of
the domains $\Omega_1$ and $\Omega_2$.  $\Gamma_h$ is the approximation of the
interface $\Gamma$. Note that for simplicity we are dropping the dependence of these
regions on the velocity field. For each element $E$ we name $e_1$ and $e_2$ its two
extremal vertices and $h_E$ its length. We let $h=\max_{E \in \Omega_h} h_E$ be the
mesh size. When more than one fracture is present, each intersection is respected by
the grid. We indicate by $(\bu_h, p_h)$ the approximation of $(\bu, p)$ for a
given mesh.

\subsection{Interface tracking algorithm}\label{subsec:interface_algorithm}

We suppose that the equations are discretized with a numerical scheme that gives an
accurate enough velocity field.
We consider an iterative scheme such that for each step $i$ a tentative
configuration $\mathcal{C}_h^{(i)}$ approximates $\mathcal{C}_h$.  For a given 
configuration $\mathcal{C}_h^{(i-1)}$, the proposed algorithm solves the differential 
problem obtaining a new velocity field $\bu^{(i)}$. To speed up the computation, we
evaluate condition \eqref{eq:law-jump} only at the extremities of each grid 
element $E$. If we obtain opposite values, we can thus determine the position of an interface in
the considered element up to a given tolerance $\epsilon_\Gamma$. This part can be coded
with an embarrassingly parallel workload, speeding up the algorithm. Having checked all 
the elements and computed the new configuration $\mathcal{C}_h^{(i)}$, the algorithm
resets with the new configuration as starting point. The exit strategy considers 
the position of $\Gamma_h$ for two successive iterations: if their distance is
smaller than a given threshold $\epsilon_\Omega$, then a stable configuration is 
reached and the algorithm ends. We summarize in Algorithm \ref{algo:interface} the implemented scheme.

\begin{algorithm}[htb]
    \SetAlgoLined
    \SetKwData{sol}{$(\bu^{(i)}, p^{(i)})$}
    \SetKwData{dom}{$\mathcal{C}_h^{(i-1)}$}
    \SetKwData{domi}{$\mathcal{C}_h^{(i)}$}
    \SetKwFunction{PDE}{PDE}\SetKwFunction{Interface}{Interface}
    \SetKwFunction{dist}{Distance}
    \KwData{For $i=1$: the configuration \dom
    and the tolerances $(\epsilon_\Gamma, \epsilon_\Omega)$\;}
    \KwResult{For $i={\rm end}$: the configuration \domi\;}
    \Do{$d > \epsilon_\Omega$}
    {
        \sol $\leftarrow$ \PDE{\dom}\;
        \For{$E=(e_1,e_2) \in \Omega_h^{(i-1)}$}
        {
            $(c_1, c_2)$ $\leftarrow$ $(\Vert \bu^{(i)}(e_1) \Vert < 1, \Vert
            \bu^{(i)}(e_2) \Vert < 1)$\;
            \If{$c_1 \neq c_2$}
            {
                \domi in $E$ $\leftarrow$ \Interface{$\bu^{(i)}$}\;
            }
        }
        $d$ $\leftarrow$ \dist{$\Gamma_h^{(i)}$, $\Gamma_h^{(i-1)}$}\;
        $i$ $\leftarrow$ $i+1$\;
    }
    \caption{Interface tracking algorithm.}
    \label{algo:interface}
\end{algorithm}
We see that with this algorithm we cannot locate multiple interfaces in a single element $E$; this is a 
direct implication of the fact that $\Omega_{1, h}^{(i)}$ or $\Omega_{2, h}^{(i)}$ would be smaller
than the mesh size in this case. The algorithm is also unable to track
interfaces which are full-dimensional (in this case of space dimension $1$) since we choose
to place only one point as the interface within an element whose vertices have velocities
around the threshold speed. 

To avoid unnecessary loops due to the chosen accuracy,
numerical experiments showed that a good value of $\epsilon_\Omega$ is
comparable with the mesh size given at the outset of the problem. Also, the
value of $\epsilon_\Gamma$ is set quite small to determine the interfaces with
high precision. Numerical examples when the interface inverse permeability jump 
is non-positive indicated convergence of the proposed scheme, independently 
from the starting configuration, thus suggesting that the 
problem has a unique solution and that, in line with Remark \ref{rem:uniqueness-interface},
the interface should not be full-dimensional. 

\subsection{Discretization in space}\label{subsec:discretization_space}

Following the algorithm discussed before, we assume that a configuration
$\mathcal{C}_h^{(i)}$ is given.  In the case of non-linear constitutive
relations $\bLambda$ an iterative scheme can be used, e.g., Picard or Newton or
L-scheme. The solution is computed up to a tolerance $\epsilon_{\texttt{nl}}$.
In our implementation we have considered a Picard iteration scheme.
For this, to discuss the numerical approximation, we assume thus a linear constitutive relation
$\bLambda$.

Following \cite{Formaggia2012}, to solve the problem we consider the mixed
finite element approximation of lowest-order degree.  For a given mesh
$\Omega_h$ we have $(\bu_h, p_h)\in
(\mathbb{P}_0(\Omega_h), \mathbb{RT}_0(\Omega_h))$, where the first is the space
of constant piecewise polynomial and the second the Raviart--Thomas space; see
\cite{Raviart1977,Roberts1991}. This pair of discrete spaces is
stable and gives a good approximation for the velocity field, essential for our
purposes. The resulting discrete
problem is well-posed, see also \cite{Formaggia2012,Boffi2013}, and the discrete solution
converges to the exact one as $h$ goes to zero.

\section{Numerical examples}\label{sec:examples}

In this part, we present numerical evidence for the quality and effectiveness of the
previously introduced framework. Particularly for increasing geometrical and physical
complexity. We consider three cases, starting from a single fracture in
Section \ref{subsec:examples:case1}, two crossing fractures in Section
\ref{subsec:examples:case2}, and concluding with the fracture network of
Benchmark 1 from \cite{Flemisch2016a} in Section \ref{subsec:examples:case3}.
In all cases, linear and non-linear
velocity-pressure relations are considered as well as the impact of a possible
vector source term $\bff$. The threshold on the velocity norm is set
as $\overline{u} = 0.15$, so that we have
\begin{gather*}
    \Omega_1 = \{ \bx \in \Omega\,|\, \norm{\bu(\bx)} < \overline{u} \}\quad
    \Omega_2 = \{ \bx \in \Omega\,|\, \norm{\bu(\bx)} > \overline{u} \}\quad
    \Gamma = \{ \bx \in \Omega\,|\, \norm{\bu(\bx)} = \overline{u} \}.
\end{gather*}
If not otherwise specified, we consider a mesh size equal to $h=5 \cdot 10^{-2}$
as well as $\epsilon_\Omega = h$; the maximum number of iterations to reach a stable
configuration with Algorithm \ref{algo:interface} is set to $50$; and to track the
interface we set $\epsilon_\Gamma = 10^{-10}$. In all the cases, the initial
configuration is chosen to be $\Omega_1^{(0)} = \Omega$, i.e., the whole domain
coincides with the low-speed region.

Our simulations, unless mentioned otherwise, are based on the following combinations of laws.
In the linear case, we consider a combination of classical Darcy laws between the velocity and
pressure, namely $\bK^{-1} \bu = -\bgrad p + \bff$ with $\bK = k \bI$, where $\bI$ is the identity matrix and
\begin{gather}
    k =
    \begin{dcases*}
        k_1 = 1 & in $\Omega_1$,\\
        k_2 = 10 & in $\Omega_2$.
    \end{dcases*}
\end{gather}
Specifically, we have
\begin{gather}\label{eq:ex_linear_case}
    \bLambda(\bu) =
    \begin{dcases*}
        \bu & in $\Omega_1$,\\
        0.1\bu & in $\Omega_2$.
    \end{dcases*}
\end{gather}
When the non-linear case is studied, a non-linear and heterogeneous relationship
between the velocity and the pressure is set. We assume a linear Darcy flow in 
$\Omega_1$ and a Darcy--Forchheimer flow in $\Omega_2$. Specifically, we have
\begin{gather}\label{eq:ex_non_linear_case}
    \bLambda(\bu) =
    \begin{dcases*}
        \bu & in $\Omega_1$,\\
        (0.01 + 3\norm{\bu}) \bu & in $\Omega_2$.
    \end{dcases*}
\end{gather}
If not specified, we consider 50 as maximum for the number of iterations 
of the non-linear solver with tolerance $\epsilon_{\texttt{nl}}$ equal to $10^{-4}$.

These examples were developed with the open source library PorePy \cite{Keilegavlen2020}.
The associated scripts are freely accessible. PorePy uses Gmsh \cite{Geuzaine2009} to construct
the grids.

\subsection{Single fracture network}\label{subsec:examples:case1}

In this first case, we consider a single fracture $\Omega=(0, 1)$. Boundary
conditions are set to zero for the pressure. In the sequel, we consider a
linear and non-linear law relationship between $\bu$ and $p$. The scalar and
vector source
terms are set equal to
\begin{gather*}
    q(\bx) =
    \begin{dcases*}
        1 & if $x_1 \leq 0.3$,\\
        -1 & if $0.3 < x_1 < 0.7$,\\
        1 & if $x_1 \geq 0.7$,
    \end{dcases*}
    \quad \text{and} \quad
    \bff = [5\cdot10^{-2}, 0, 0]^\top
    .
\end{gather*}

\subsubsection{Linear case}\label{subsubsec:examples:case1:linear}

In this part, we consider the linear case with $\bLambda$ as in \eqref{eq:ex_linear_case}.
The numerical solution is reported in Figure \ref{fig:ex1_linear} along with some snapshots of
the tentative solutions from Algorithm \ref{algo:interface}.
What the ``if $\Omega_1$'' legend represents is a binary outcome saying if the region
is $\Omega_1$ or not. The "condition" legend represents the configuration at the previous
algorithm iteration. The first figure gives the initial condition imposed, 
the second the configuration after 3 steps and the third the final
solution (at iteration 6). We notice the creation of multiple interfaces
$\Gamma$, which might change during the determination of the final configuration.
\begin{figure}[btp]
    \centering
    \includegraphics[width=1\textwidth]{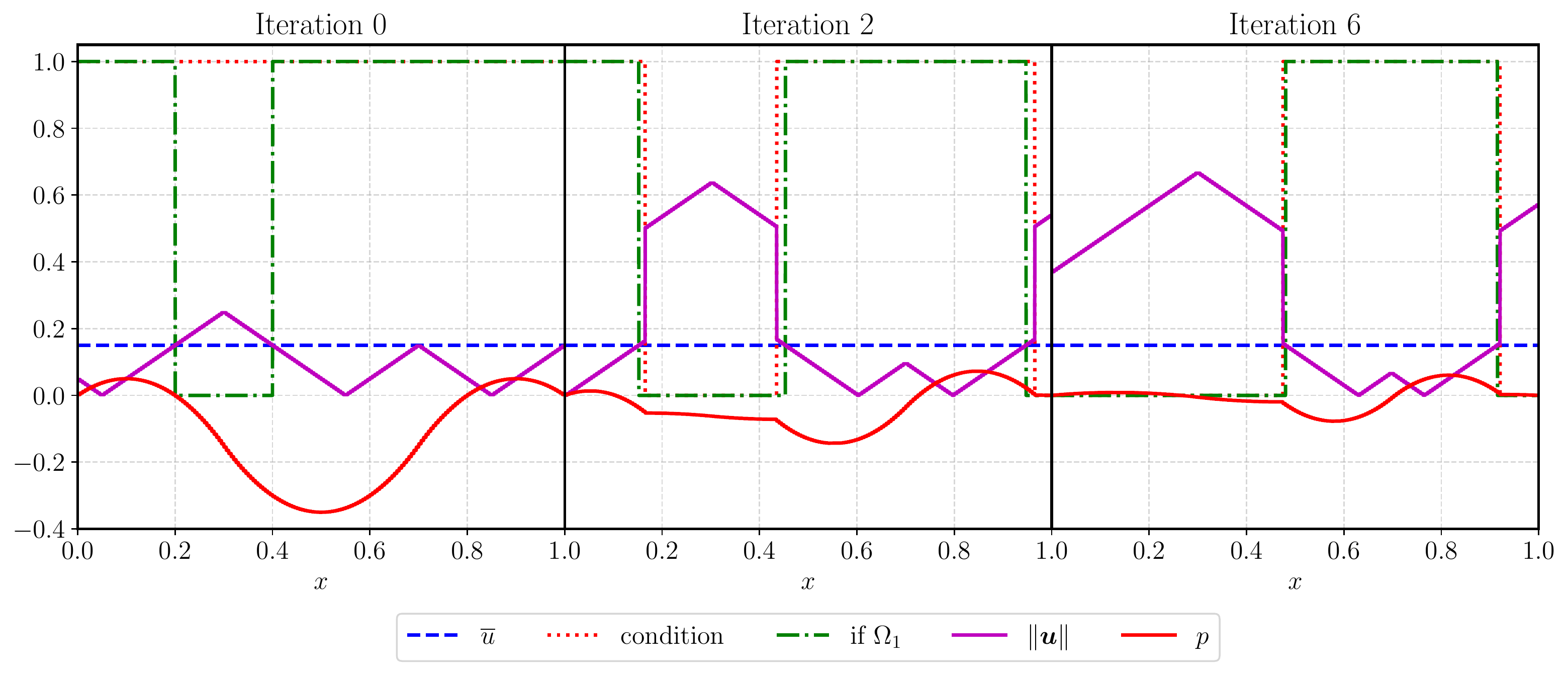}%
    \caption{Solution for different iterations for the problem of Section
    \ref{subsubsec:examples:case1:linear}. From the left at iteration 0, 2, and
    6. The green line represents if the
    portion of the fracture belongs to $\Omega_1$ or not. The pressure profile
    is amplified by a factor of 10.}%
    \label{fig:ex1_linear}
\end{figure}
By changing the initial condition we get to the same stationary solution, and
by refining the grid we obtain a stable outcome similar to the one presented in
Figure \ref{fig:ex1_linear}.

In Figures \ref{fig:ex1_linear_k2} and \ref{fig:ex1_linear_k2_two_states} we study
a case different from \eqref{eq:ex_linear_case}, where we test our algorithm
for various permeabilities $k_2$ while fixing $k_1 =1$. By increasing 
the maximum number of iterations to $1000$, Figure \ref{fig:ex1_linear_k2} on the
left shows the impact of changing $k_2$ on the number of iterations for
the Algorithm \ref{algo:interface}. We see that for $k_2 \geq k_1$ the solution
is always computable while we cannot draw the same conclusion for $k_2 < k_1$. 
In this latter case, the algorithm ``jumps'' between the two states and thus does not converge. 
An example, for $k_2=0.5625$ is shown in Figure \ref{fig:ex1_linear_k2_two_states}, 
where we notice that ``if $\Omega_1$'' and ``condition'' are perfectly flipped in any two
successive iterations.
\begin{figure}[btp]
    \centering
    \includegraphics[width=0.45\textwidth]{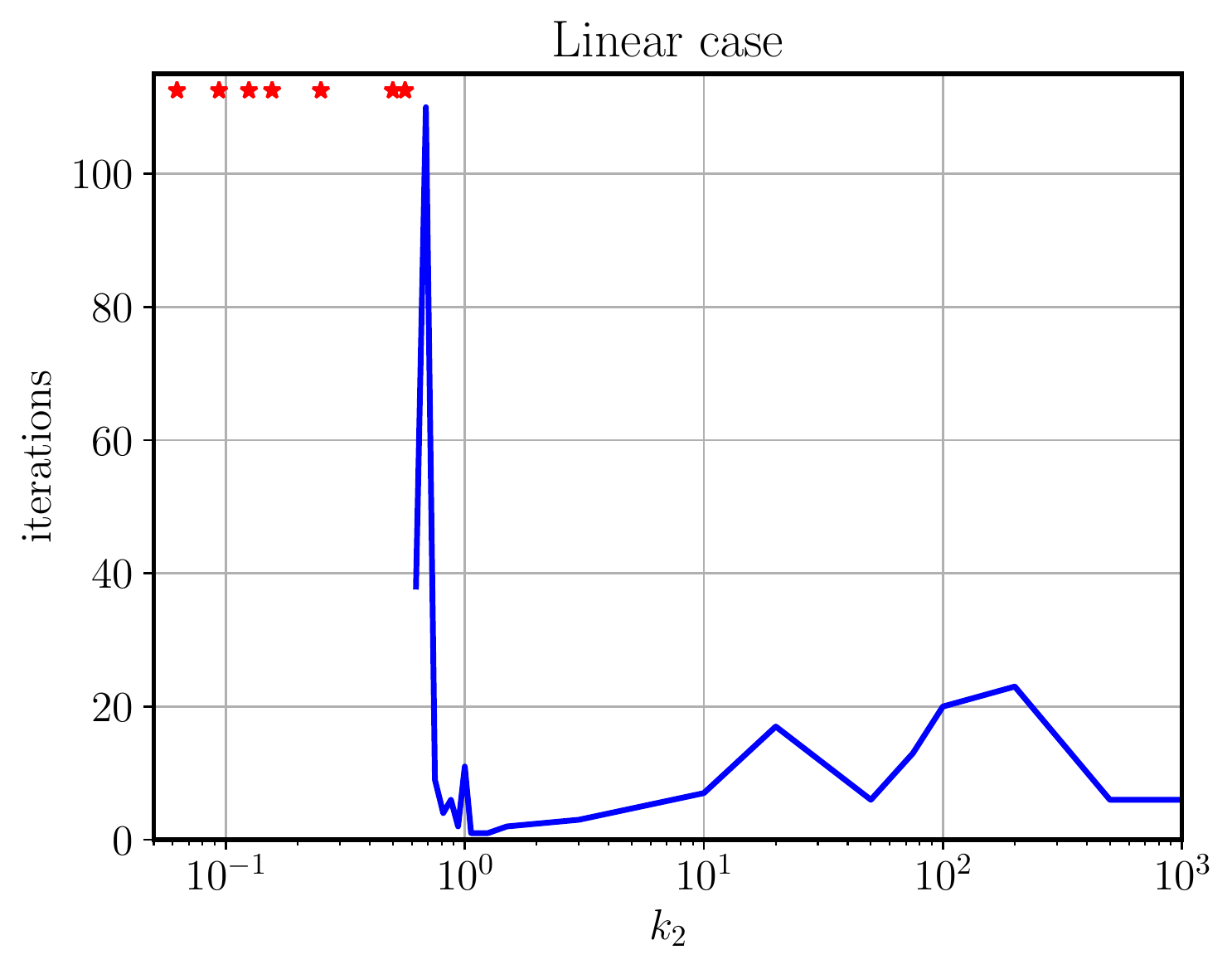}%
    \hspace*{0.1\textwidth}%
    \includegraphics[width=0.45\textwidth]{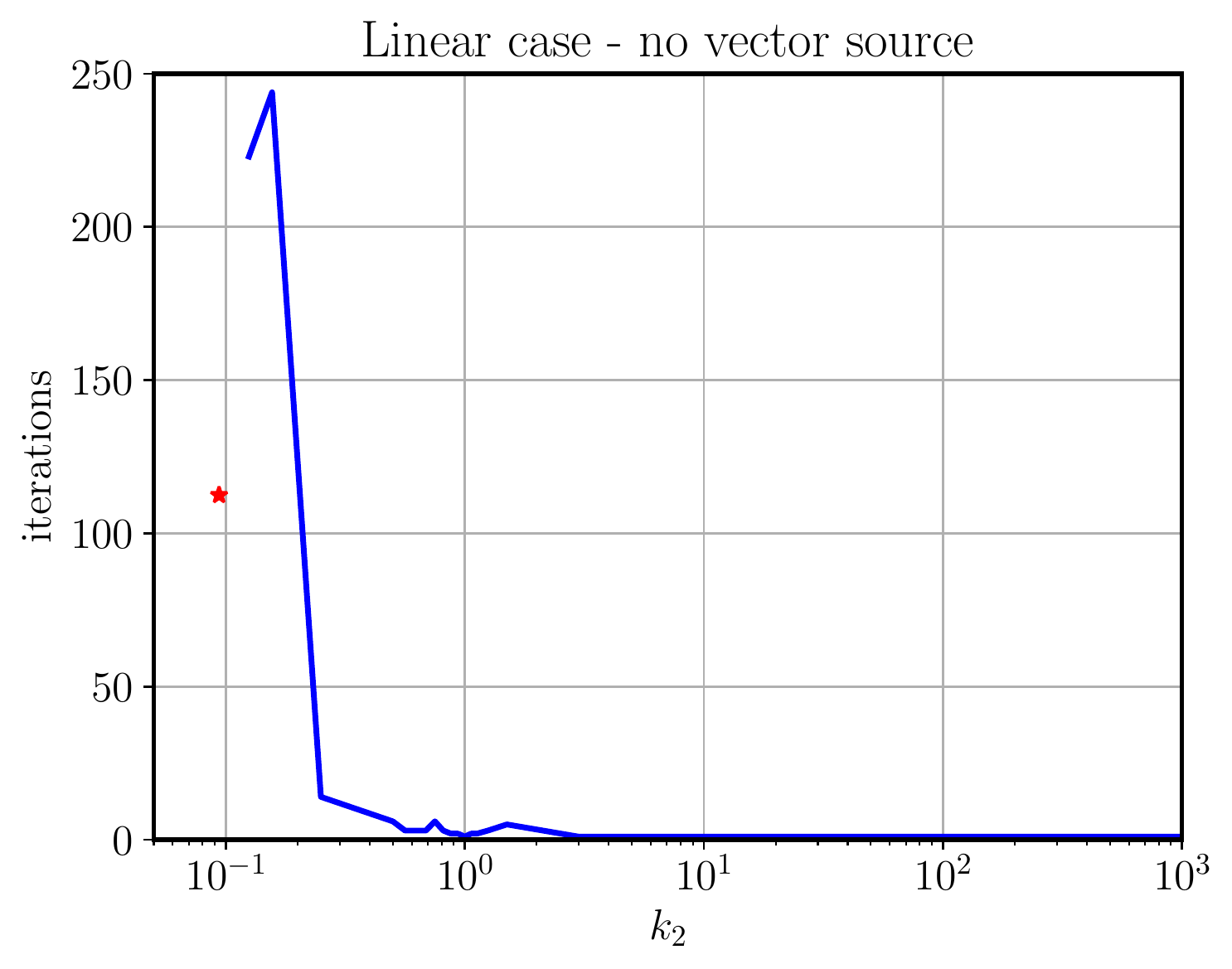}
    \caption{Number of iterations of Algorithm \ref{algo:interface} for different values of
    $k_2$ for the example in Section
    \ref{subsubsec:examples:case1:linear}. The red asterisks mean that
    the maximum number of iterations is reached.}%
    \label{fig:ex1_linear_k2}
\end{figure}

\begin{figure}[tbp]
    \centering
    \includegraphics[width=0.66\textwidth]{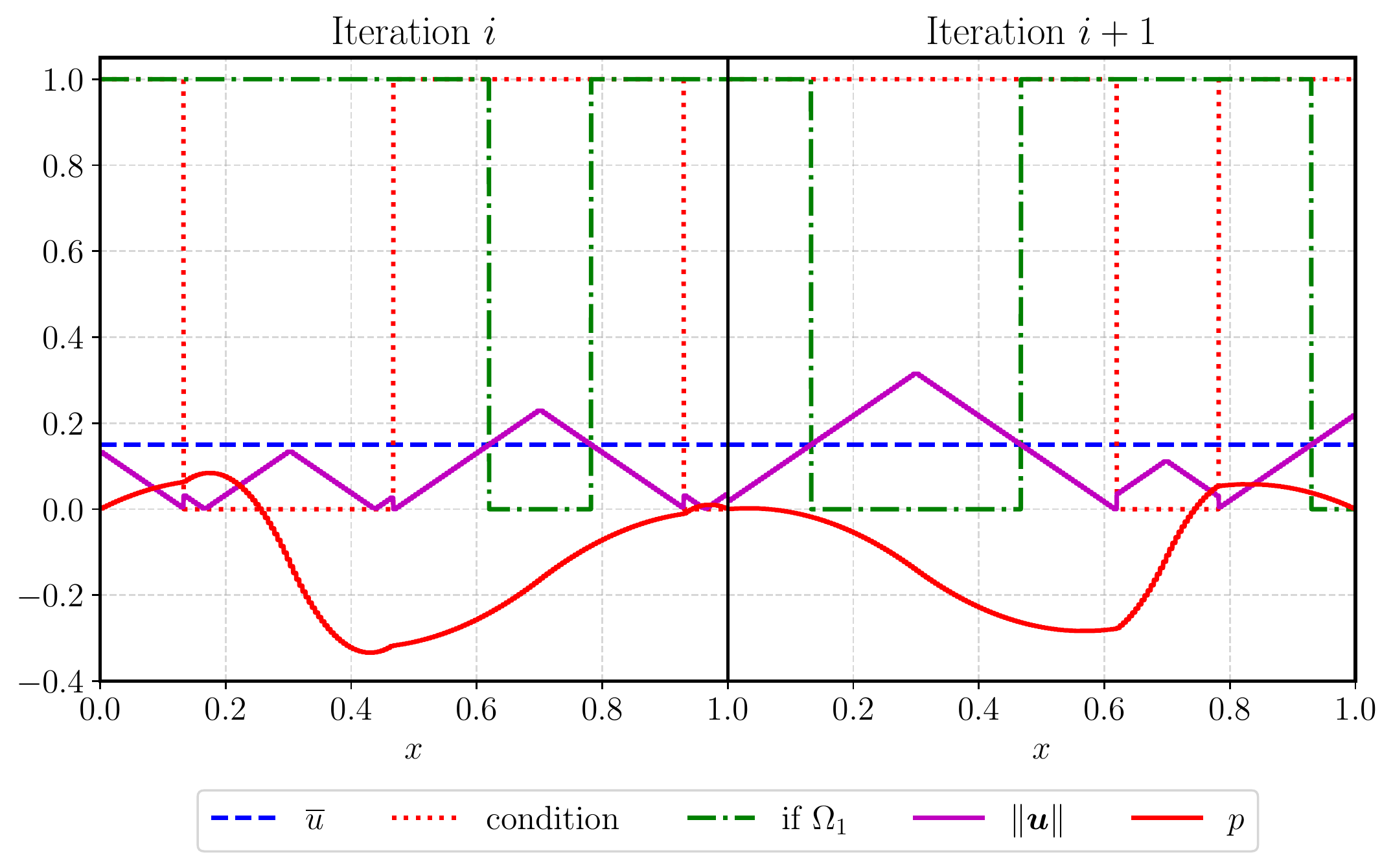}%
    \caption{The two solutions computed by the algorithm for $k_2=0.5625$; see Section \ref{subsubsec:examples:case1:linear}.}
    \label{fig:ex1_linear_k2_two_states}
\end{figure}

By setting $\bff = \bnull$ and $p(1)=0.2$, the latter to avoid that the algorithm converges
in 1 iteration for each value of $k_2$, we can do the same analysis and obtain
the plot in Figure \ref{fig:ex1_linear_k2} on the right. We deduce that
the presence of the vector source has an impact on the computability of
the solution when $k_2>k_1$, that is, when the interface permeability jump
is positive.

We can conclude that, even in this simple setting, the obtained numerical
evidence is interesting and gives a valid support to the developed
theory, at least when $k_2\geq k_1$.

\subsubsection{Non-linear case}\label{subsubsec:examples:case1:non_linear}

In this part, we consider the non-linear case with $\bLambda$ as in \eqref{eq:ex_non_linear_case}.
Figure \ref{fig:ex1_non_linear} shows the numerical solution
obtained at different iteration steps of Algorithm \ref{algo:interface}.
\begin{figure}[btp]
    \centering
    \includegraphics[width=0.66\textwidth]{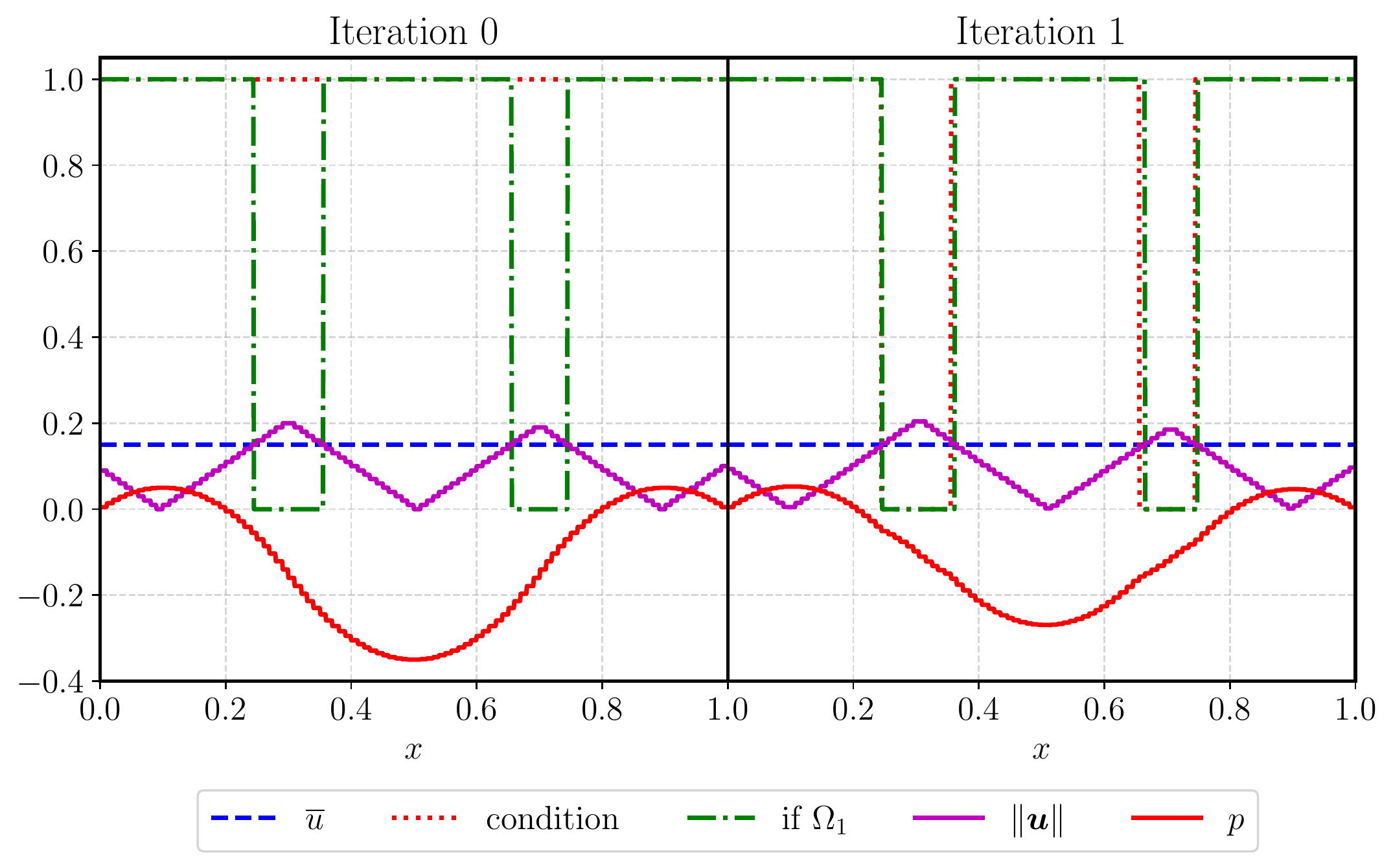}%
    \caption{Solution for different iterations for the problem of Section
    \ref{subsubsec:examples:case1:non_linear}. From the left at iteration 0, 1. The green line represents if the
    portion of the fracture belongs to $\Omega_1$ or not. The pressure profile
    is amplified by a factor of 100.}%
    \label{fig:ex1_non_linear}
\end{figure}
The stable solution is reached very quickly and only two iterations of the outer scheme are needed. We see the effect on the pressure of the non-linear law, which changes shape between the initial configuration and iteration 1.
By changing the parameters, it is possible to show that the
obtained solution is independent from the initial
condition and stable with respect the grid refinement, once the mesh size is
small enough to separate close interfaces. In this case 2 iterations are needed
to reach the stable solution, the first requires only 1 iteration and the second 6.
\begin{table}[tbp]
    \centering
    \begin{tabular}{ |c|c|c|c|c| }
        \hline
        $\epsilon_{\texttt{nl}}$ & \texttt{it\textsubscript{out}} &
        \texttt{it\textsubscript{in}} &
        \texttt{err\textsubscript{p}} &
        \texttt{err\textsubscript{u}}
        \\ \hline
        $10^{-1}$ &  2 & 3  & $1.2\cdot10^{-5}$ & $8.7\cdot10^{-6}$\\ \hline
        $10^{-2}$ &  2  & 4  & $1.8\cdot10^{-7}$ & $1.3\cdot10^{-7}$ \\ \hline
        $10^{-3}$ &  2  & 5  & $2.8\cdot10^{-9}$ & $2.2\cdot10^{-9}$ \\ \hline
        $10^{-4}$ &  2  & 6 & $4.4\cdot10^{-11}$ & $2.1\cdot10^{-11}$ \\ \hline
        $10^{-5}$ &  2  & 7 & $2.4\cdot10^{-12}$ & $1.0\cdot10^{-11}$ \\ \hline
        $10^{-8}$ & 2 & 11 & $0$ & $0$ \\
        \hline
        $10^{-12}$ & 2 & 15 & $-$ & $-$ \\ \hline
    \end{tabular}
    \caption{Numbers of iterations and errors computed for the example in
    Section \ref{subsubsec:examples:case1:non_linear}.}
    \label{tab:ex1_non_linear}
\end{table}

We consider now the effect of the tolerance imposed in the non-linear solver
$\epsilon_{\texttt{nl}}$, in
particular its effect on the number of iterations and resulting error. By
keeping fixed the spatial discretization, we compute a reference solution
$(\bu_{\rm ref}, p_{\rm ref})$ with
tolerance $\epsilon_{\texttt{nl}} = 10^{-12}$.
We report in Table \ref{tab:ex1_non_linear} the comparison with higher
tolerances. In all cases the first iteration requires only two non-linear cycles, since the initial configuration has only a linear problem.
At the second iteration the non-linear steps depend on the chosen tolerance, this value is reported in the table.
In particular, we notice that both errors
\texttt{err\textsubscript{p}} for $p$ and \texttt{err\textsubscript{u}} for
$\bu$ computed as
\begin{gather*}
    \text{\texttt{err\textsubscript{p}}} = \frac{\norm{p_{\rm ref} -
    p}}{\norm{p_{\rm ref}}}
    \quad \text{and} \quad
    \text{\texttt{err\textsubscript{u}}} = \frac{\norm{\bu_{\rm ref} -
    \bu}}{\norm{\bu_{\rm ref}}}
\end{gather*}
have a monotone decay. The norms in the previous expression are the Euclidean norms
of the solution vector.
The error is rather small and decays very quickly, reaching zero for the non-linear tolerance equal to $10^{-3}$. The outer iterations are not influenced by this parameters, probably due to the small errors obtained in all the cases.

Also in this case we can conclude that, even in this simple setting, the obtained numerical
evidence is interesting and gives a valid support to the developed
theory.

\subsection{Crossing fractures}\label{subsec:examples:case2}

In this second case, we consider the domain made of two crossing fractures. We
set
$\Omega = \Omega_{\rm horiz} \cup \Omega_{\rm vert}$ where $\Omega_{\rm horiz} =
(0, 1) \times \{0.5\} $ and $\Omega_{\rm vert} =  \{0.5\} \times (0, 1)$; both $\Omega_{\rm horiz}$
and $\Omega_{\rm vert}$ are identifiable with $(0,1)$. We consider a zero vector source
term $\bff$ and a scalar source term given on $\Omega_{\rm horiz}$ and $\Omega_{\rm
vert}$ respectively by
\begin{gather*}
    q_{\rm horiz}(\bx) =
    \begin{dcases*}
        1 & if $x_1 \leq 0.3$\\
        -1 & if $0.3 < x_1 < 0.7$\\
        1 & if $x_1 \geq 0.7$
    \end{dcases*}
    \quad \text{and} \quad
    q_{\rm vert}(\bx) =
    \begin{dcases*}
        1 & if $x_2 \leq 0.3$\\
        -1 & if $0.3 < x_2 < 0.7$\\
        1 & if $x_2 \geq 0.7$
    \end{dcases*}.
\end{gather*}
On the boundary we set $p(0, 0.5) = 0$ and $p(0.5, 0) = p(1,0.5) = p(0.5, 1) = 0.1$.

\subsubsection{Linear case}\label{subsubsec:examples:case2:linear}

In this part, we consider the linear case with $\bLambda$ as in \eqref{eq:ex_linear_case}.
Figure \ref{fig:ex2_linear} shows the
graphical representation of the solution for both the horizontal and vertical
part of $\Omega$ for all the iterations of Algorithm \ref{algo:interface}.
\begin{figure}[btp]
    \centering
    \includegraphics[width=1\textwidth]{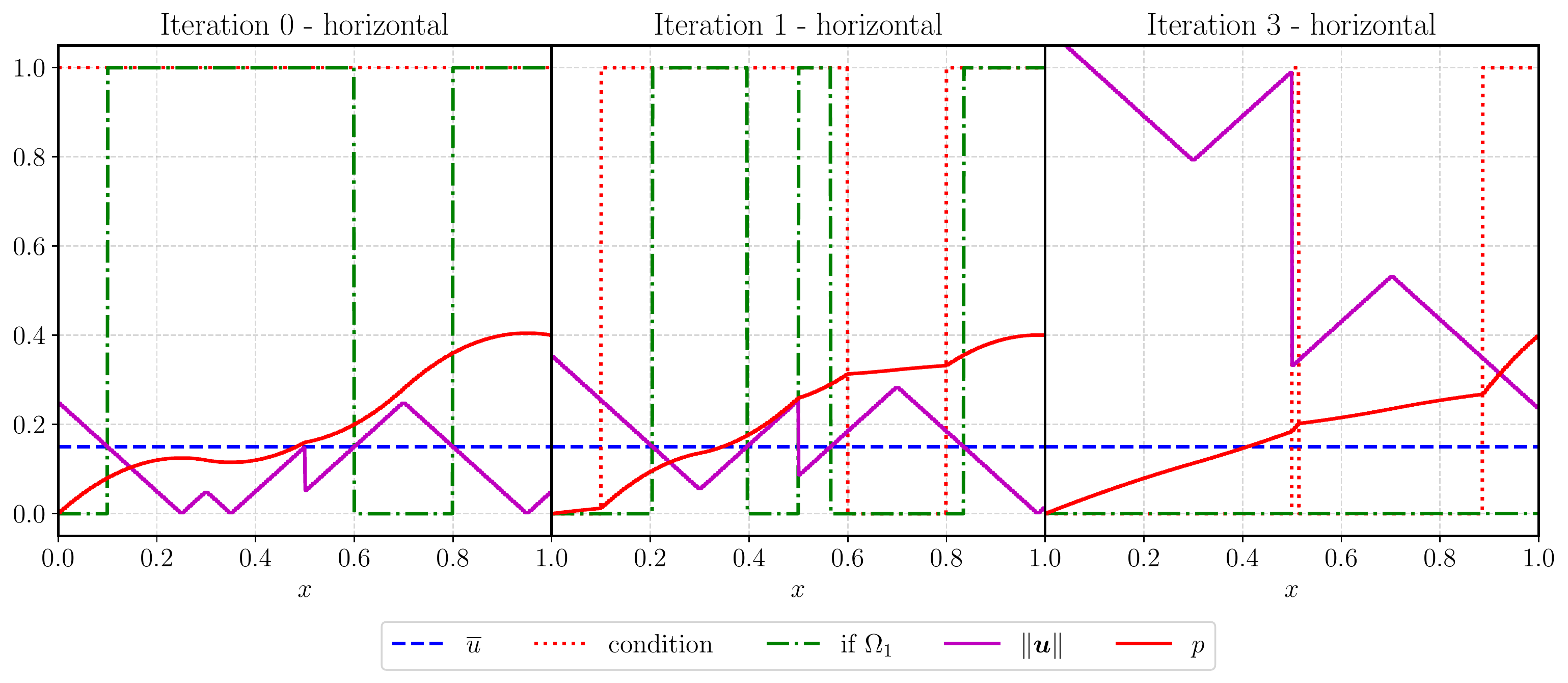}\\
    \includegraphics[width=1\textwidth]{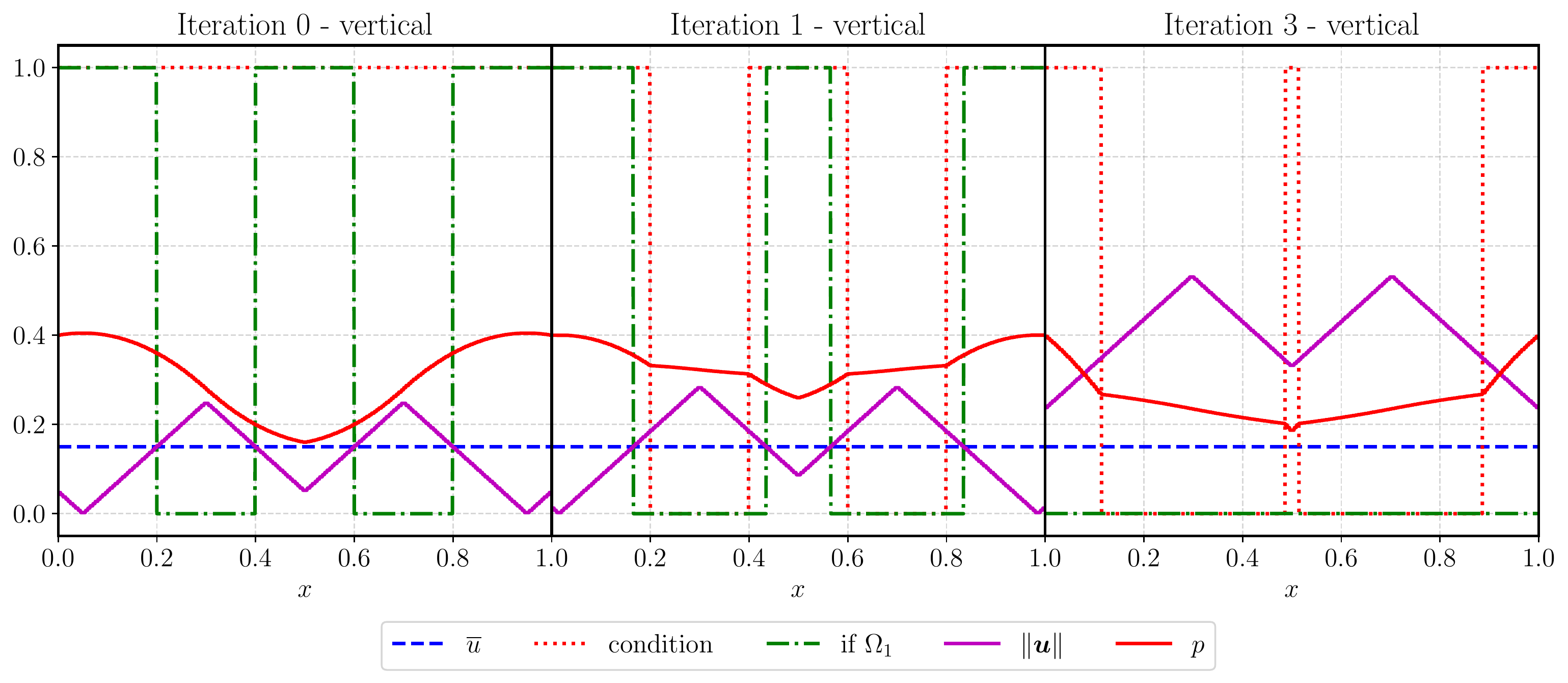}
    \caption{Solution for different iterations for the problem of Section
    \ref{subsubsec:examples:case2:linear}. The pressure profile
    is amplified by a factor of 4. On the top for the horizontal part of
    $\Omega$, while on the bottom for the vertical one.}%
    \label{fig:ex2_linear}
\end{figure}
We notice the influence of the crossing through a velocity jump in $\Omega_{\rm
horiz}$, while the pressure profile is continuous as condition \eqref{eq:network_cc} 
imposes. Also in this case, by changing the
initial condition we obtain the same final outcome, where all the domain becomes
$\Omega_{2,h}$.

Also with this more complex case, the obtained numerical
evidence is insightful and shows good properties for the developed approximation
framework which is apparently applicable to fracture networks.

\subsubsection{Non-linear case}\label{subsubsec:examples:case2:nonlinear}

In this part, we consider the non-linear case with $\bLambda$ as in
\eqref{eq:ex_non_linear_case}.
The obtained
numerical solution is reported in Figure \ref{fig:ex2_nonlinear}. The scheme
takes 5 iterations to converge with increasing number of non-linear solver iterations as $(1, 6, 8, 9)$.
\begin{figure}[btp]
    \centering
    \includegraphics[width=1\textwidth]{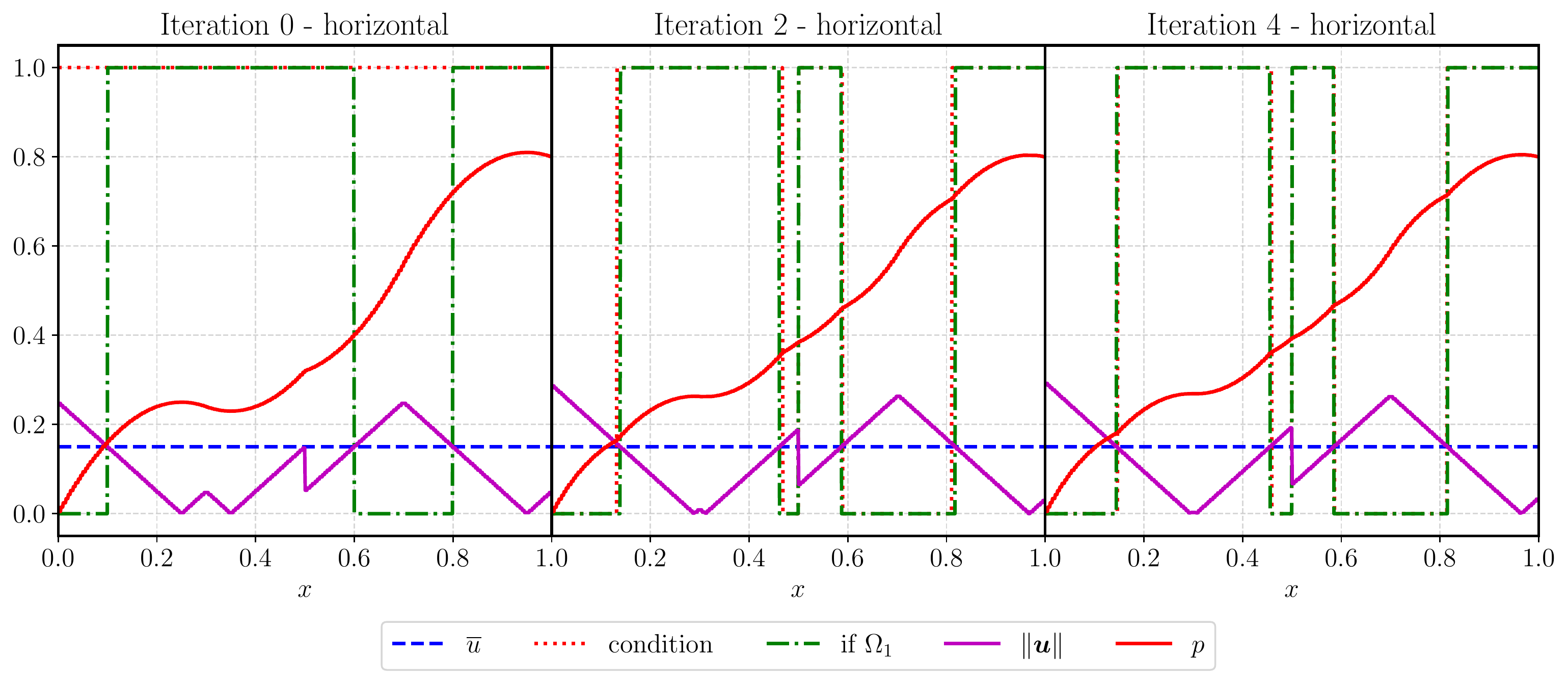}\\
    \includegraphics[width=1\textwidth]{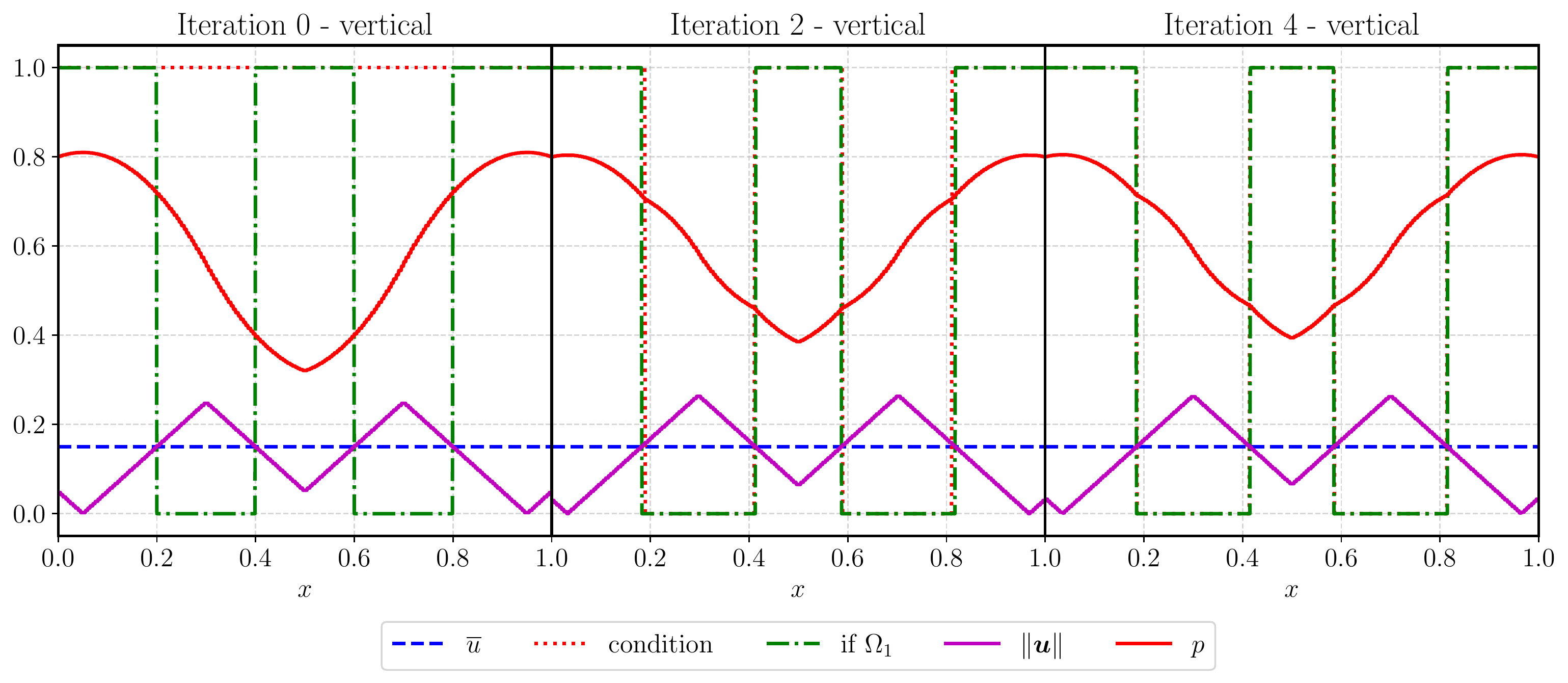}
    \caption{Solution for different iterations for the problem of Section
    \ref{subsubsec:examples:case2:nonlinear}. The pressure profile
    is amplified by a factor of 8. On the top for the horizontal part of
    $\Omega$, while on the bottom for the vertical one.}%
    \label{fig:ex2_nonlinear}
\end{figure}
The obtained solution shows that the high-speed model, being Darcy--Forchheimer, is more 
proper to describe most of the problem leaving the slow Darcian regime in the vicinity of the boundary where the non-zero pressure
condition is imposed. It is important to note that the plots show the norm of $\bu$ which presents a jump only in the horizontal fracture. Nevertheless, condition \eqref{eq:network_cc} is respected at the fracture intersection since a velocity jump is also present in the other fracture. The representation of only $\norm{\bu}$ hides this details.
By changing
the initial condition or refining the mesh, we obtain again the same final outcome.

We can conclude that also in presence of a non-linear and heterogeneous law the
proposed framework works properly on a network of fractures.

\subsection{Multiple fracture network}\label{subsec:examples:case3}

We finally consider a complex fracture network, with geometry taken from
Benchmark 1 of \cite{Flemisch2016a}. It is composed of 6 intersecting fractures
as Figure \ref{fig:ex3_domain} shows, along with the set of boundary conditions.
\begin{figure}[tbp]
    \centering
    \resizebox{0.3\textwidth}{!}{\fontsize{2.5cm}{2cm}\selectfont
\begingroup%
  \makeatletter%
  \providecommand\color[2][]{%
    \errmessage{(Inkscape) Color is used for the text in Inkscape, but the package 'color.sty' is not loaded}%
    \renewcommand\color[2][]{}%
  }%
  \providecommand\transparent[1]{%
    \errmessage{(Inkscape) Transparency is used (non-zero) for the text in Inkscape, but the package 'transparent.sty' is not loaded}%
    \renewcommand\transparent[1]{}%
  }%
  \providecommand\rotatebox[2]{#2}%
  \newcommand*\fsize{\dimexpr\f@size pt\relax}%
  \newcommand*\lineheight[1]{\fontsize{\fsize}{#1\fsize}\selectfont}%
  \ifx\svgwidth\undefined%
    \setlength{\unitlength}{282.82656097bp}%
    \ifx\svgscale\undefined%
      \relax%
    \else%
      \setlength{\unitlength}{\unitlength * \real{\svgscale}}%
    \fi%
  \else%
    \setlength{\unitlength}{\svgwidth}%
  \fi%
  \global\let\svgwidth\undefined%
  \global\let\svgscale\undefined%
  \makeatother%
  \begin{picture}(1,1)%
    \lineheight{1}%
    \setlength\tabcolsep{0pt}%
    \put(0,0){\includegraphics[width=\unitlength,page=1]{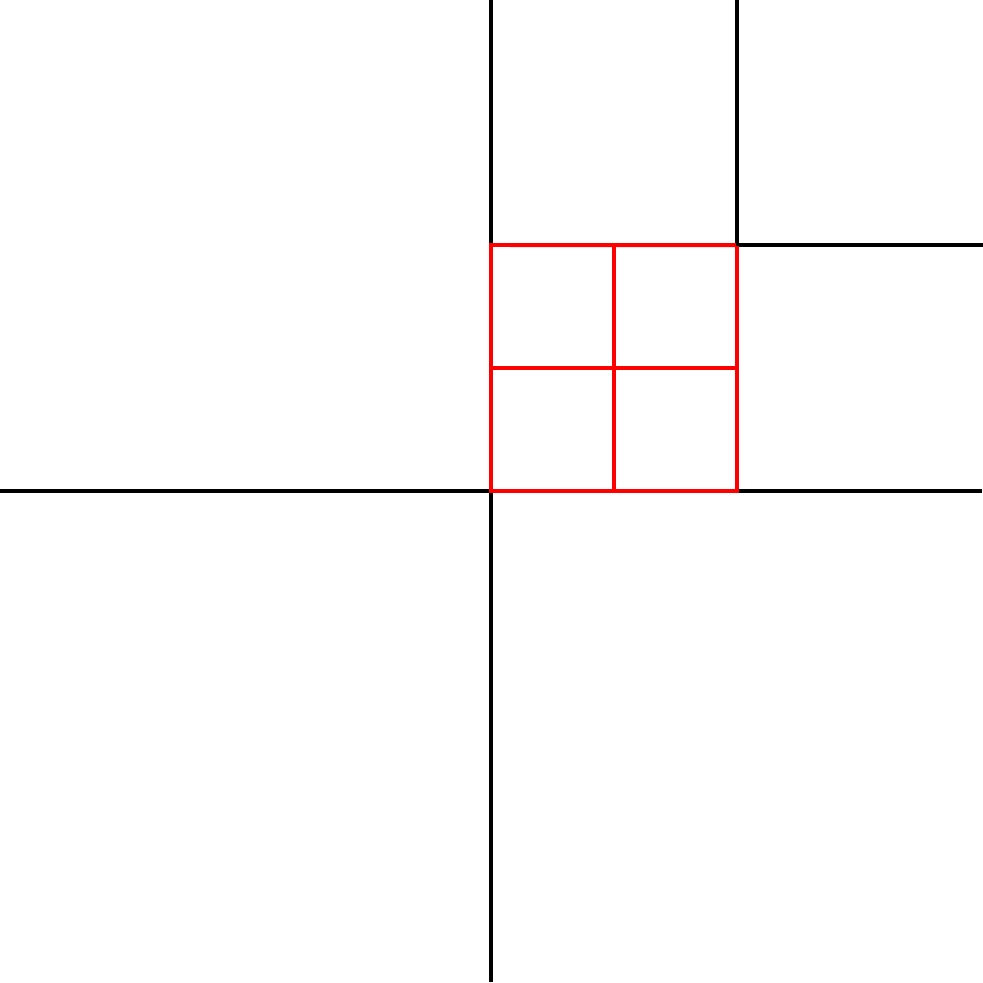}}%
    \put(0.60669763,0.40905868){\color[rgb]{1,0,0}\makebox(0,0)[lt]{\lineheight{1.25}\smash{\begin{tabular}[t]{l}$\Omega_{s}$\end{tabular}}}}%
    \put(0.52001825,0.01852142){\color[rgb]{0,0,0}\makebox(0,0)[lt]{\lineheight{1.25}\smash{\begin{tabular}[t]{l}$p=0$\end{tabular}}}}%
    \put(0.00847476,0.52280191){\color[rgb]{0,0,0}\makebox(0,0)[lt]{\lineheight{1.25}\smash{\begin{tabular}[t]{l}$p=0$\end{tabular}}}}%
    \put(0.93600384,0.52775225){\color[rgb]{0,0,0}\makebox(0,0)[lt]{\lineheight{1.25}\smash{\begin{tabular}[t]{l}$p=0.2$\end{tabular}}}}%
    \put(0.93600384,0.78584685){\color[rgb]{0,0,0}\makebox(0,0)[lt]{\lineheight{1.25}\smash{\begin{tabular}[t]{l}$p=0.2$\end{tabular}}}}%
    \put(0.76664569,0.96262232){\color[rgb]{0,0,0}\makebox(0,0)[lt]{\lineheight{1.25}\smash{\begin{tabular}[t]{l}$p=0.2$\end{tabular}}}}%
    \put(0.22321823,0.96262232){\color[rgb]{0,0,0}\makebox(0,0)[lt]{\lineheight{1.25}\smash{\begin{tabular}[t]{l}$p=0.2$\end{tabular}}}}%
  \end{picture}%
\endgroup%
}
    \caption{Representation of the fracture network for the examples in
    Section \ref{subsec:examples:case3}. We have reported the boundary conditions as well as the
    portion of the network with small branches in red $\Omega_s$.}%
    \label{fig:ex3_domain}
\end{figure}
We denote by $\Omega_s$ the set of smaller fracture branches, which will be useful 
in the following. Null and unitary vector and scalar sources are considered, respectively.
As done with the previous examples, we will consider the linear and non-linear case
in the next Sections.

\subsubsection{Linear case}\label{subsubsec:examples:case3:linear}

We consider here the same linear relations as in Section 
\ref{subsubsec:examples:case1:linear}; see \eqref{eq:ex_linear_case}. 
The obtained solution is represented in Figure \ref{fig:ex3_linear}. 
The scheme converges after 4 iterations of Algorithm \ref{algo:interface}.
We see an interesting result: $\Omega_{2,h}$, which is the high-velocity 
region, is automatically positioned on the main pathways between
the inflow and outflow parts of the network.
\begin{figure}[btp]
    \centering
    \includegraphics[width=0.3\textwidth]{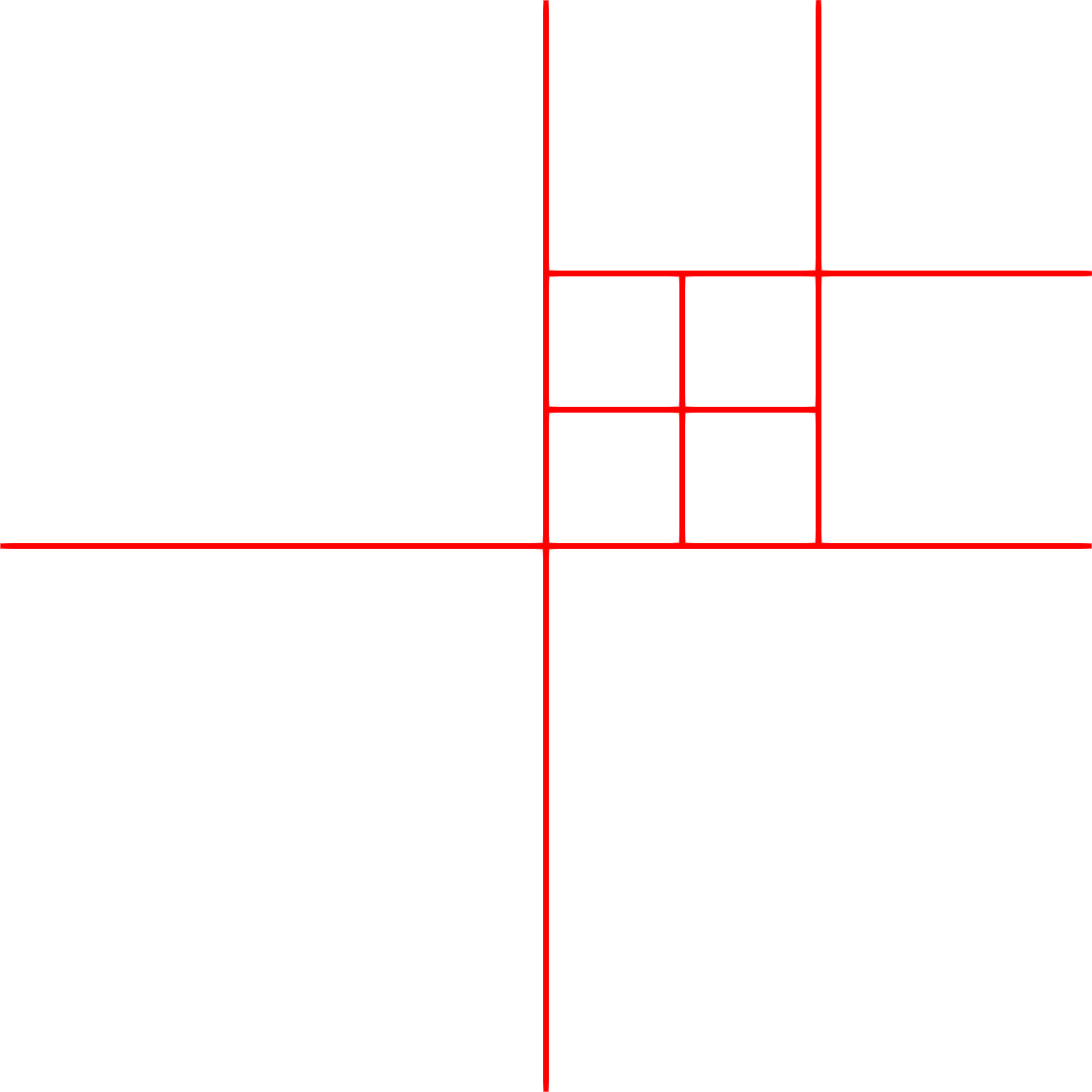}%
    \hspace*{0.05\textwidth}%
    \includegraphics[width=0.3\textwidth]{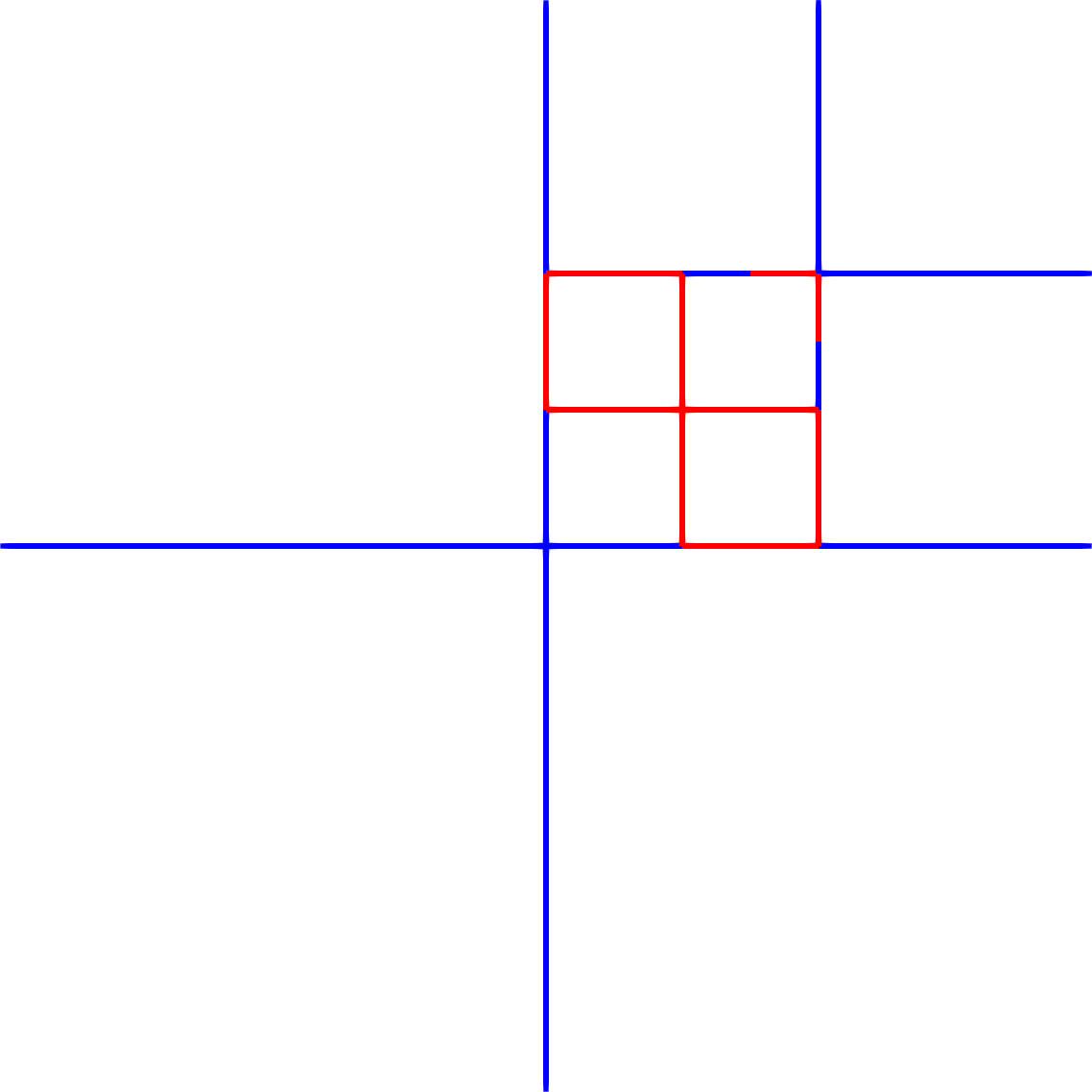}%
    \hspace*{0.05\textwidth}%
    \includegraphics[width=0.3\textwidth]{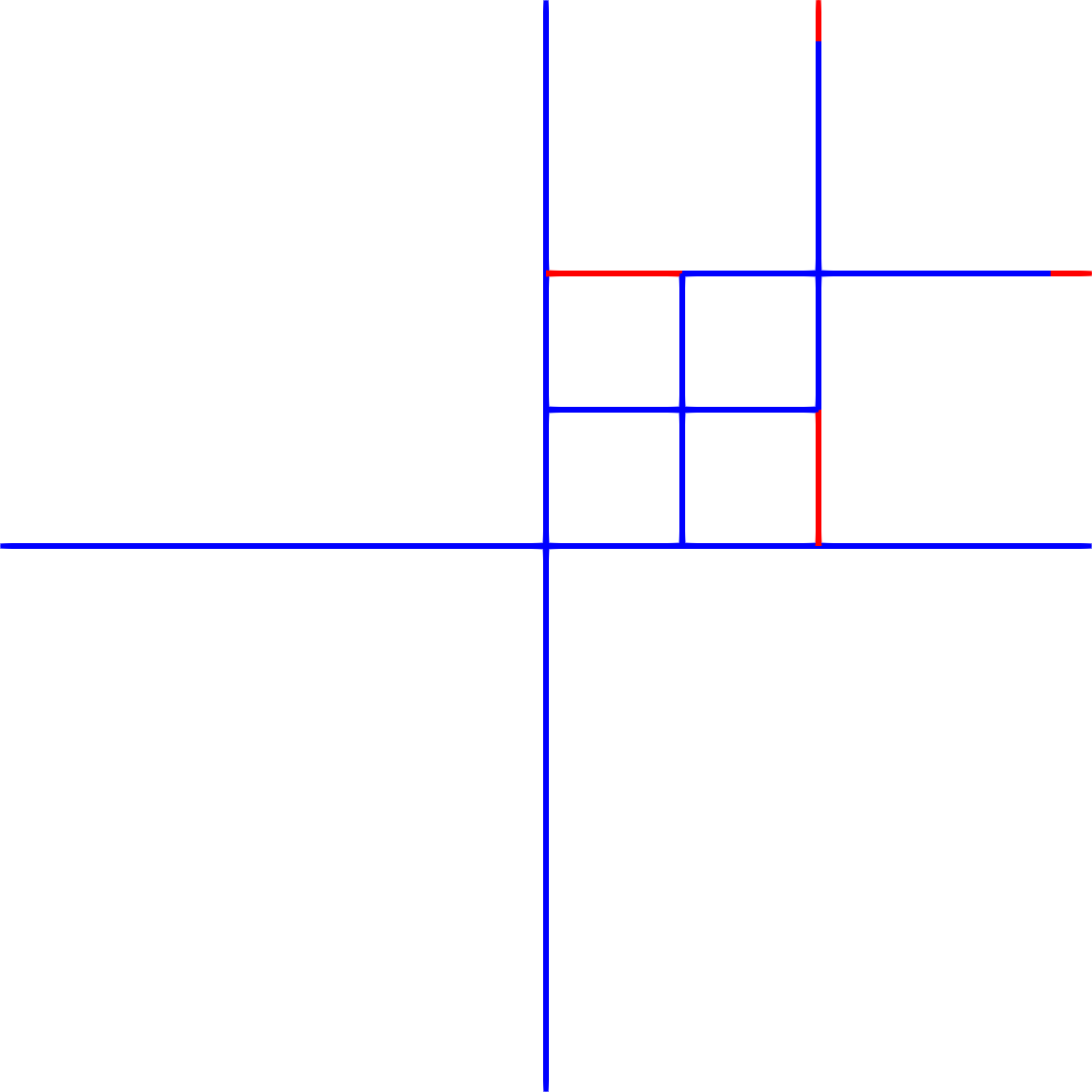}\\[0.05\textwidth]
    \includegraphics[width=0.3\textwidth]{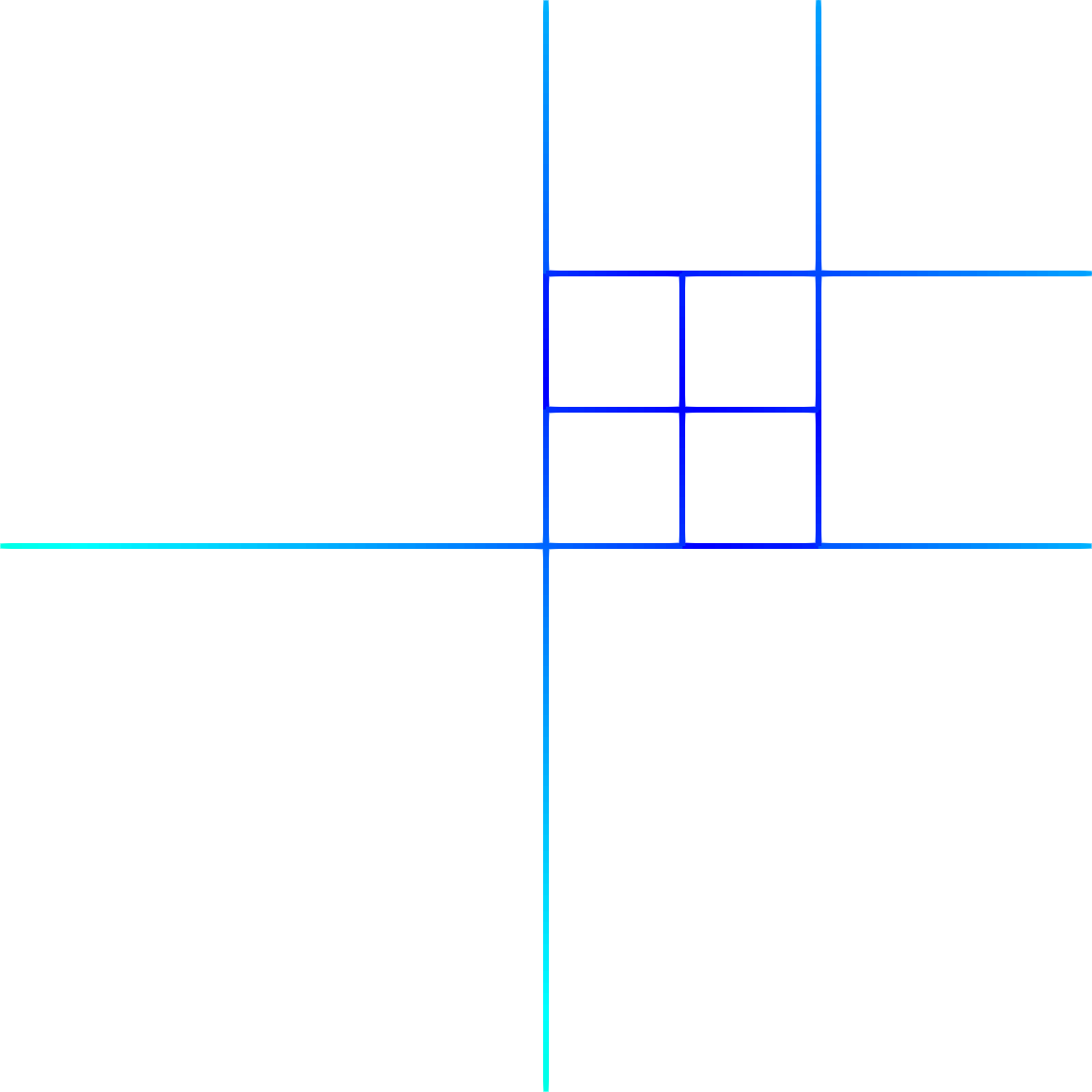}%
    \hspace*{0.05\textwidth}%
    \includegraphics[width=0.3\textwidth]{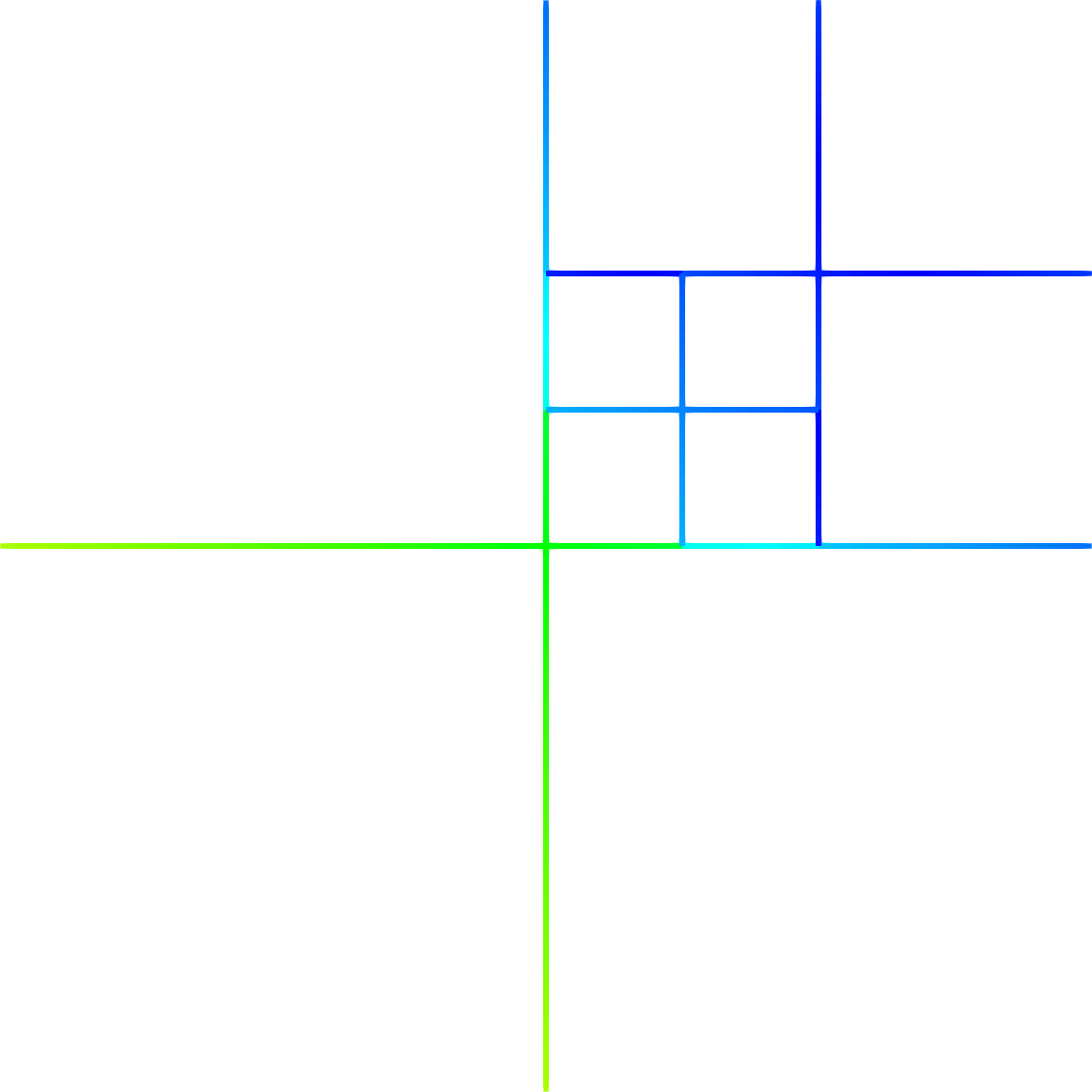}%
    \hspace*{0.05\textwidth}%
    \includegraphics[width=0.3\textwidth]{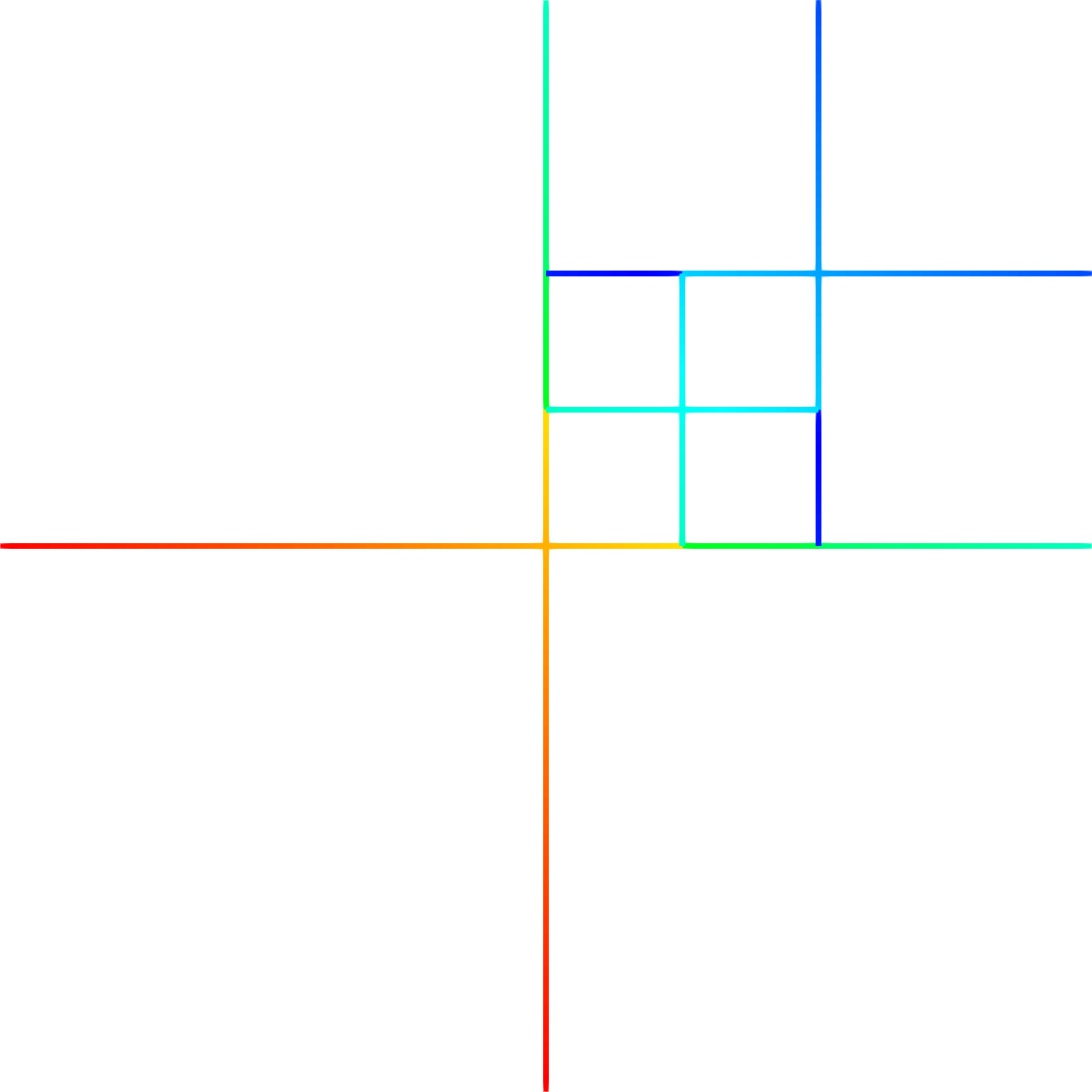}\\[0.025\textwidth]
    \includegraphics[width=0.3\textwidth]{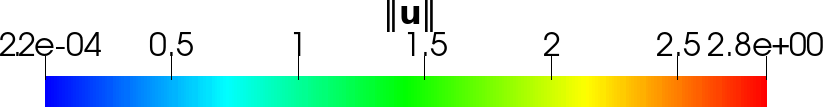}
    \caption{Solution for different iterations, respectively $(0, 1, 4)$,
    for the problem of Section \ref{subsubsec:examples:case3:linear}.
    On the top the condition ``if $\Omega_1$'' is represented with red
    indicating "true" and blue "false". On the bottom the norm of the
    velocity.}%
    \label{fig:ex3_linear}
\end{figure}
Again, the algorithm showed robustness when changing the initial configuration and
refining the grid.

Even for this complex configuration, the proposed algorithm is capable to compute a
reasonable solution with a limited cost.

\subsubsection{Non-linear case}\label{subsubsec:examples:case3:nonlinear}

We consider in this part the non-linear case, where the constitutive law combination is
given as
\begin{gather*}
    \bLambda(\bu) =
    \begin{dcases*}
        \bu & in $\Omega_1$,\\
        (0.01 + 0.25\norm{\bu}) \bu & in $\Omega_2$.
    \end{dcases*}
\end{gather*}
Algorithm \ref{algo:interface} takes 4 steps to reach the final configuration with an
average of $18$ iterations of the non-linear solver for each step.
Figure \ref{fig:ex3_non_linear} shows the obtained numerical solution for
different iterations.
\begin{figure}[btp]
    \centering
    \includegraphics[width=0.3\textwidth]{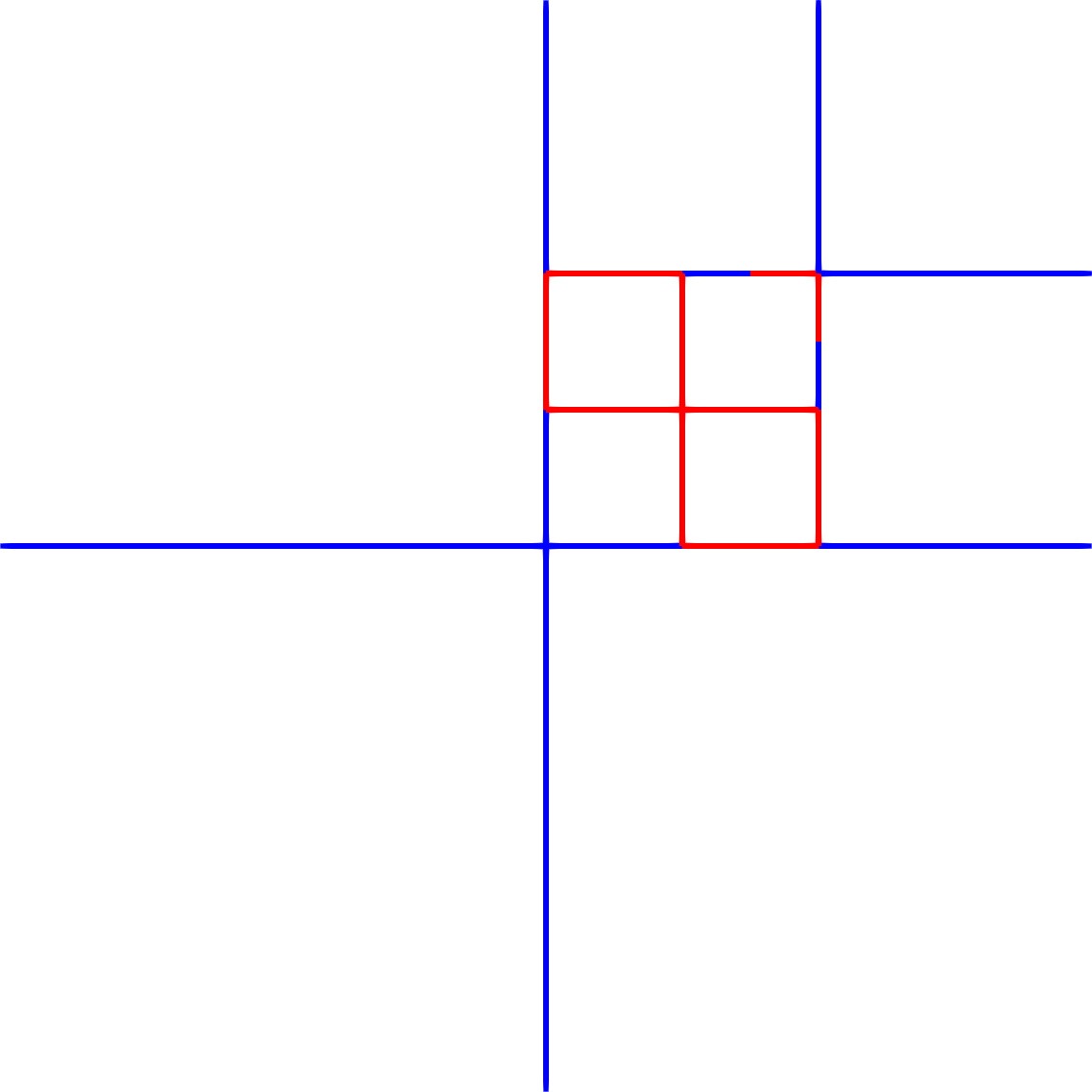}%
    \hspace*{0.05\textwidth}%
    \includegraphics[width=0.3\textwidth]{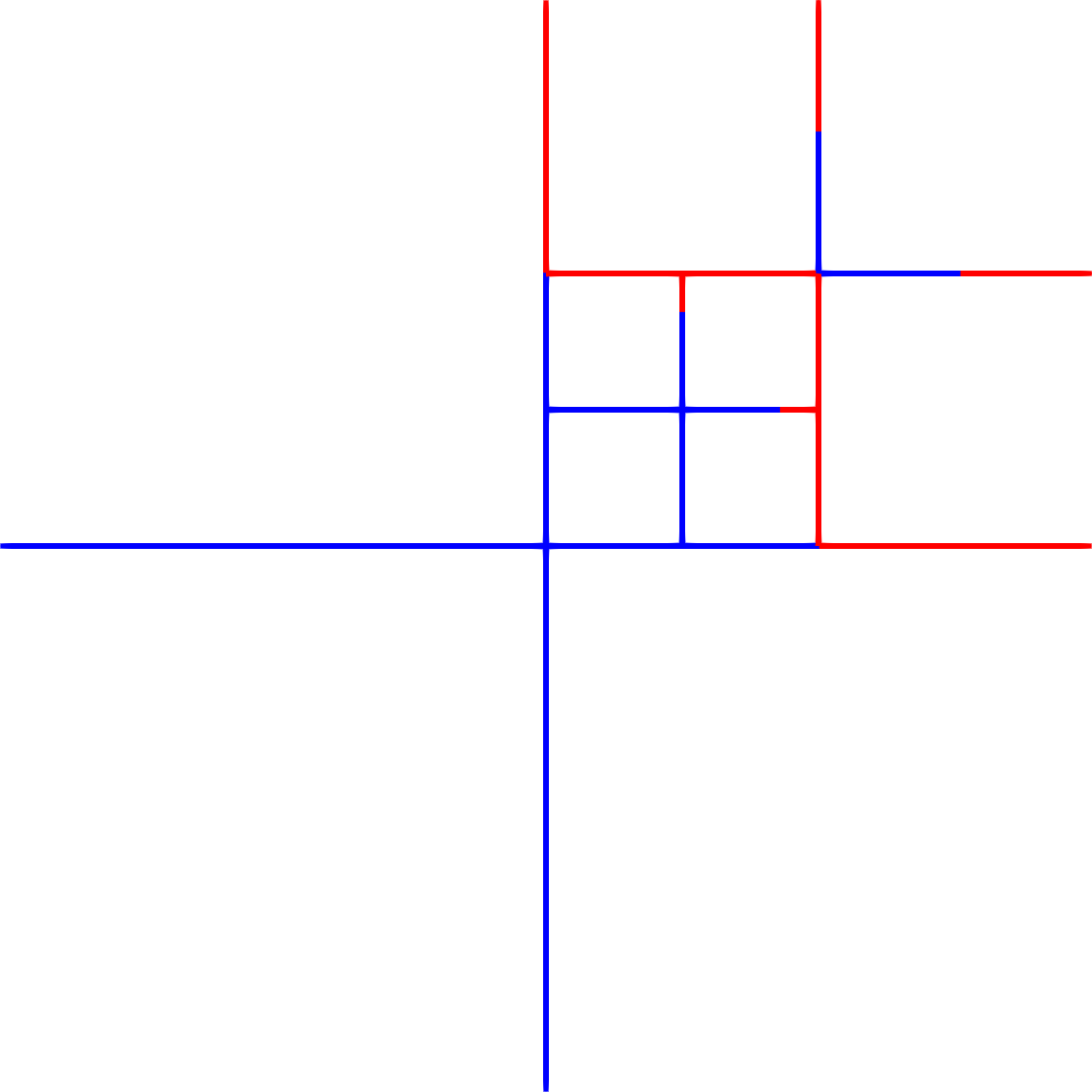}%
    \hspace*{0.05\textwidth}%
    \includegraphics[width=0.3\textwidth]{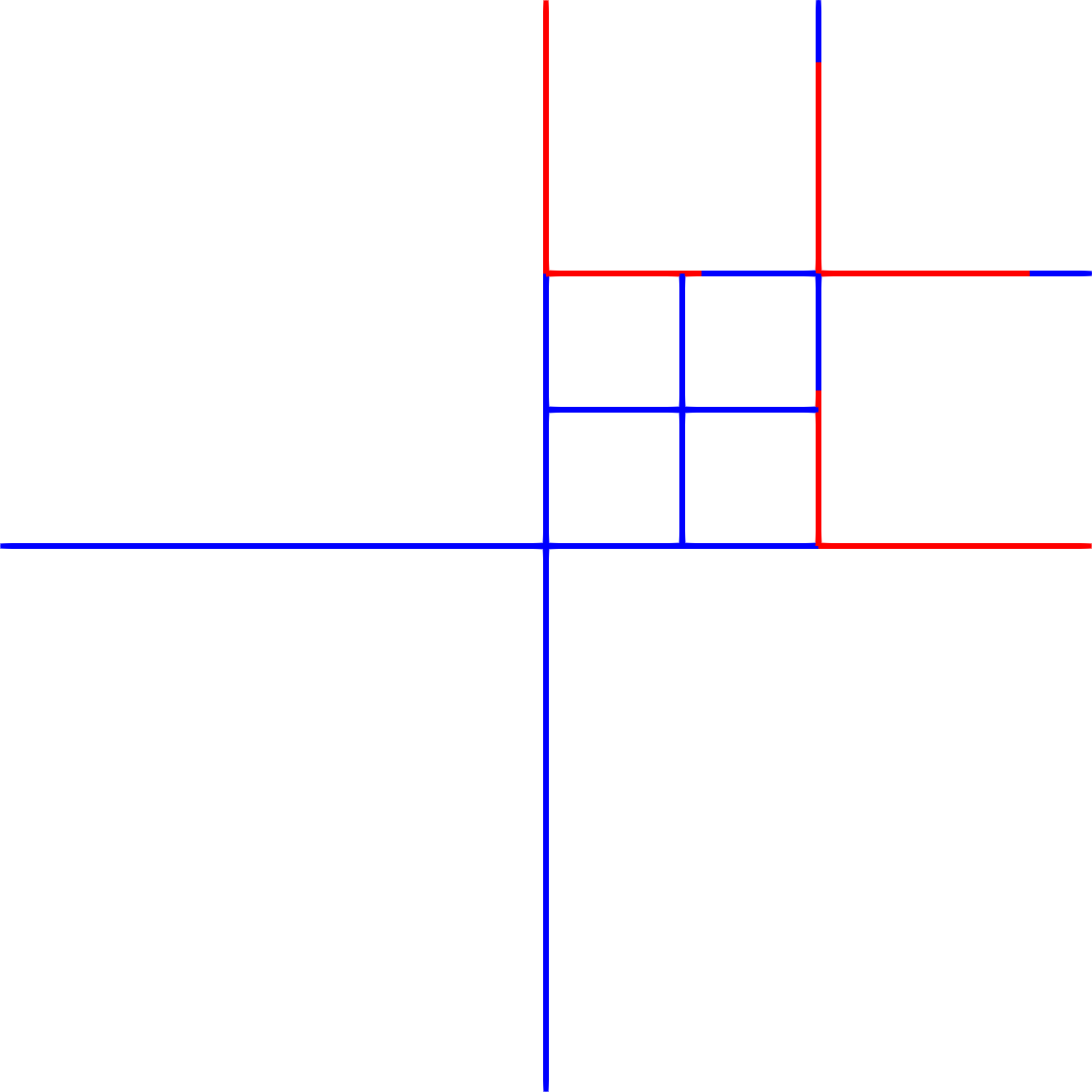}\\[0.05\textwidth]
    \includegraphics[width=0.3\textwidth]{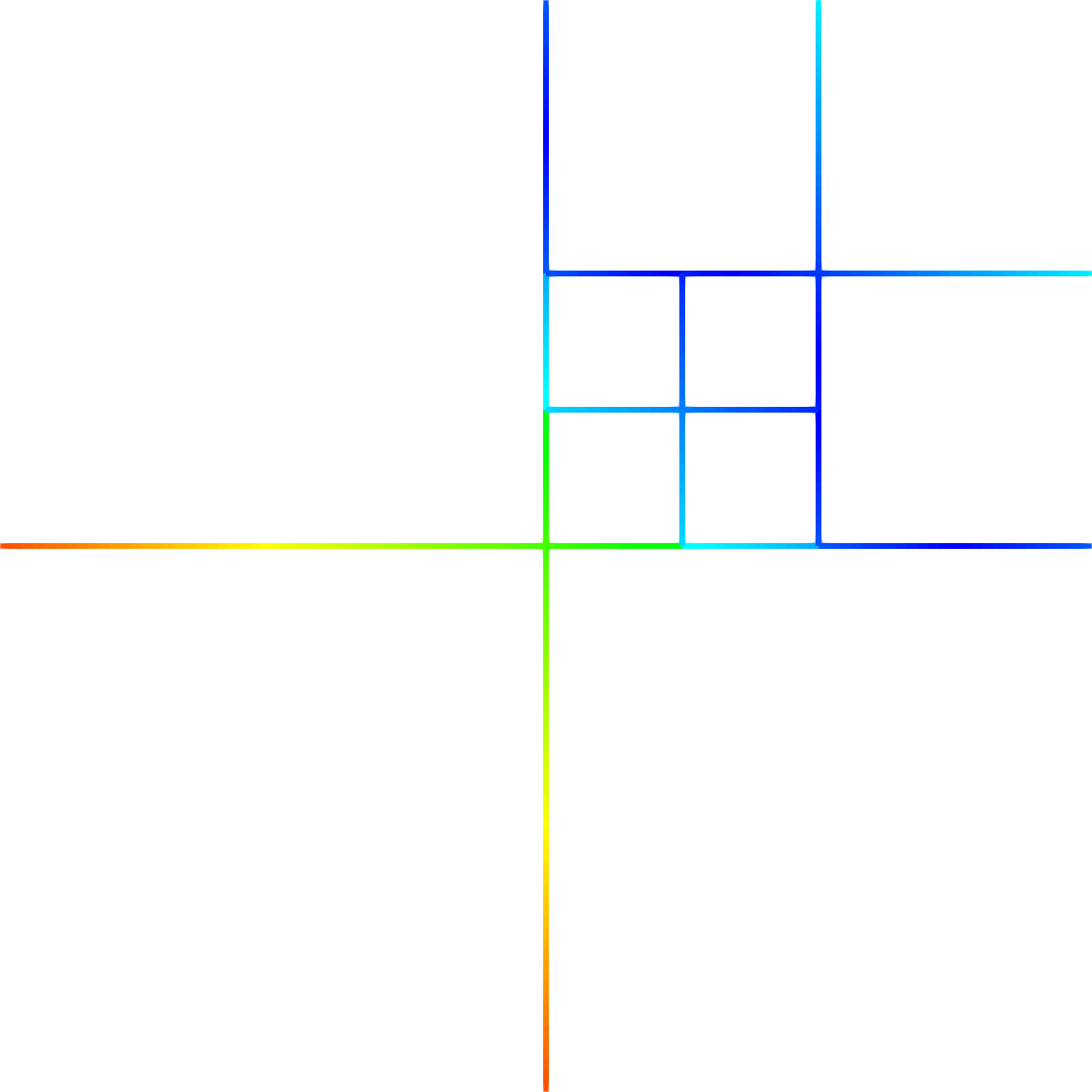}%
    \hspace*{0.05\textwidth}%
    \includegraphics[width=0.3\textwidth]{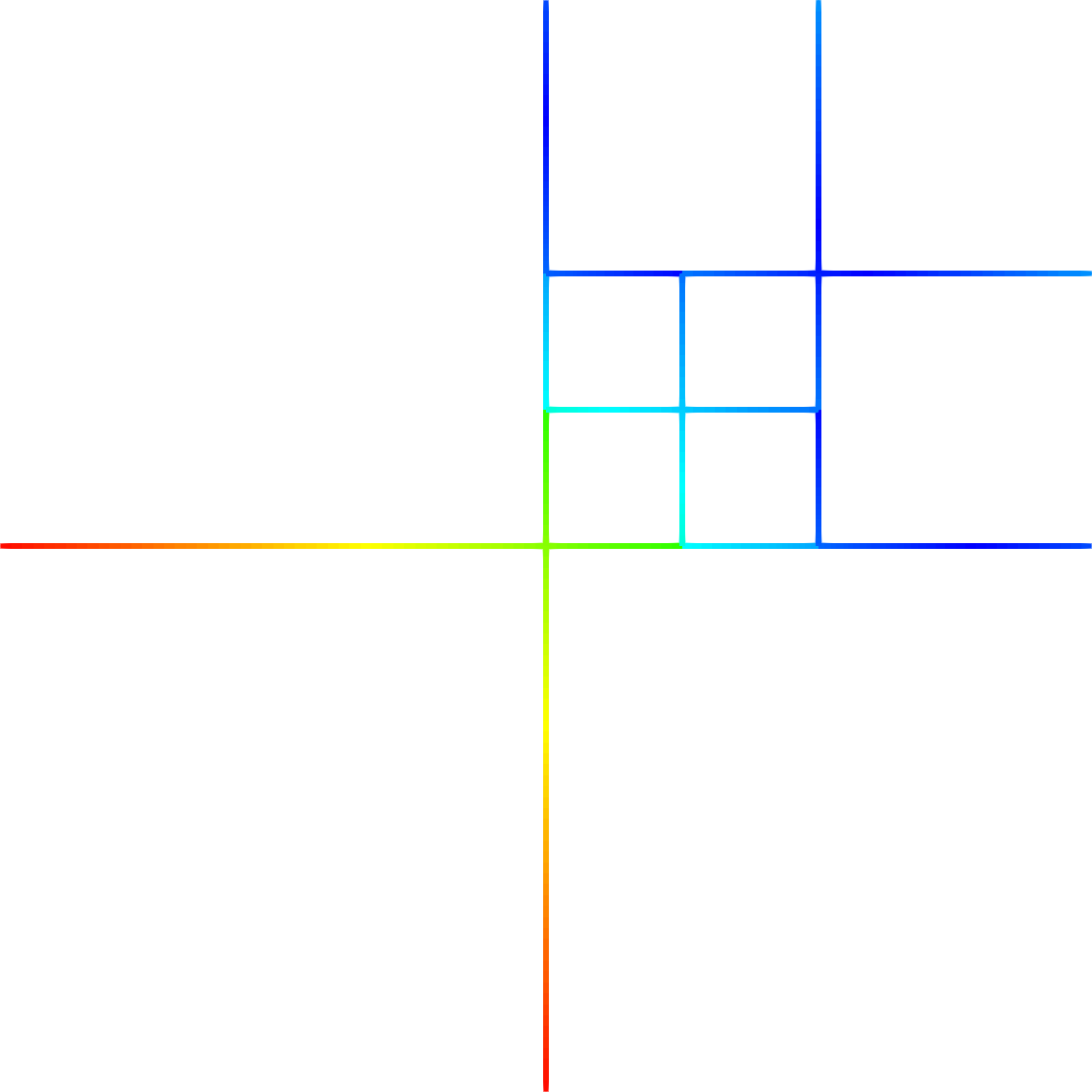}%
    \hspace*{0.05\textwidth}%
    \includegraphics[width=0.3\textwidth]{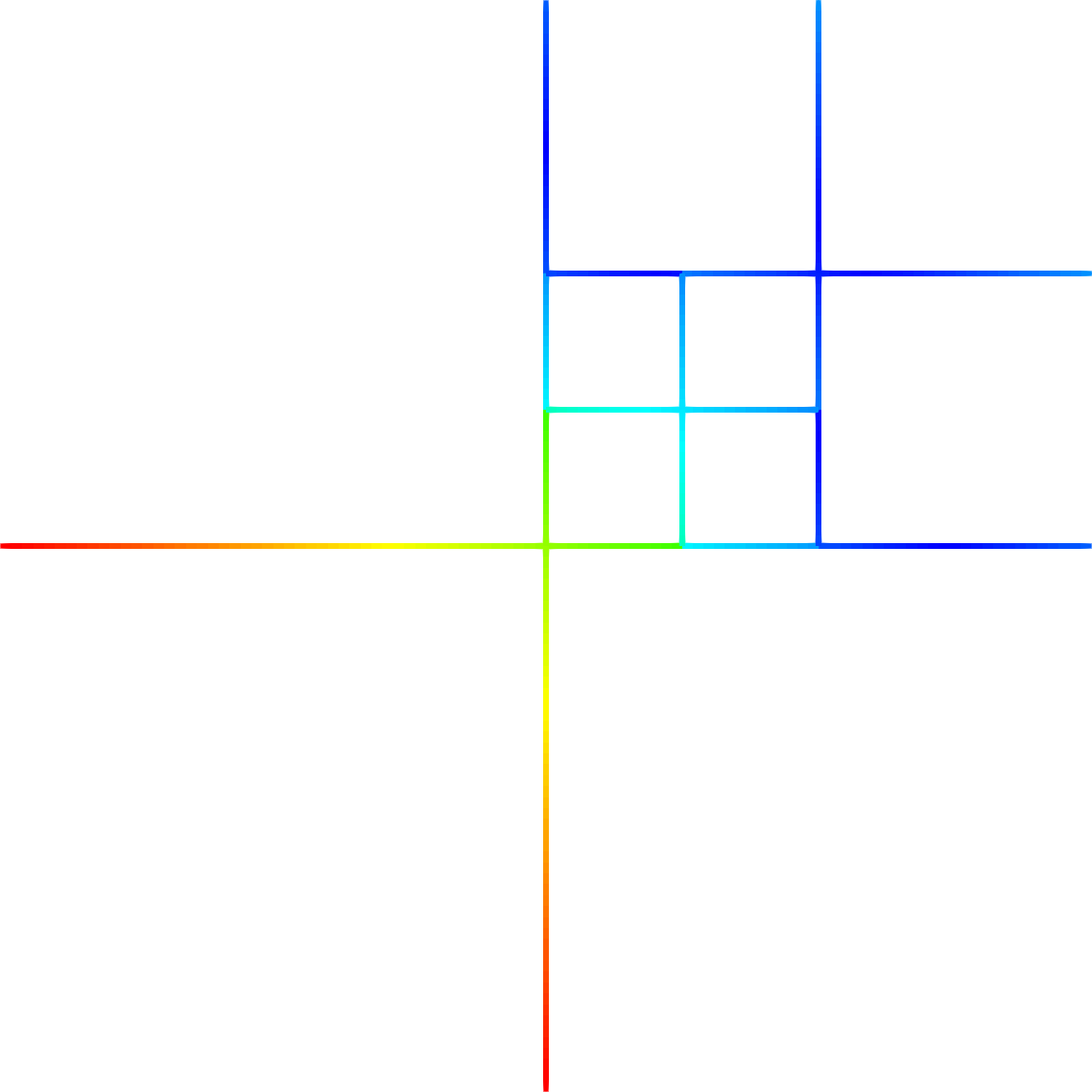}\\[0.025\textwidth]
    \includegraphics[width=0.3\textwidth]{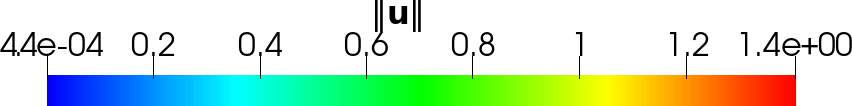}
    \caption{Solution for different iterations, respectively $(1, 3, 4)$,
    for the problem of Section \ref{subsubsec:examples:case3:nonlinear}.
    On the top the condition ``if $\Omega_1$'' is represented with red
    indicating "true" and blue "false". On the bottom the norm of the
    velocity.}
    \label{fig:ex3_non_linear}
\end{figure}
The obtained solution is again insightful, positioning the high-speed region,
given by $\Omega_{2, h}$, in the longest fracture branches and the low-speed
region $\Omega_{1, h}$ mainly in the fracture
branches at the outflow. There is a transition zone from $\Omega_{2, h}$ to $\Omega_{1,
h}$ that mainly takes place in $\Omega_s$. A zoom-in is reported in Figure \ref{fig:ex3_non_linear_zoom} to better clarify the evolution of $\Omega_{2, h}$ and $\Omega_{1,
h}$ at the small fracture branches.
\begin{figure}[btp]
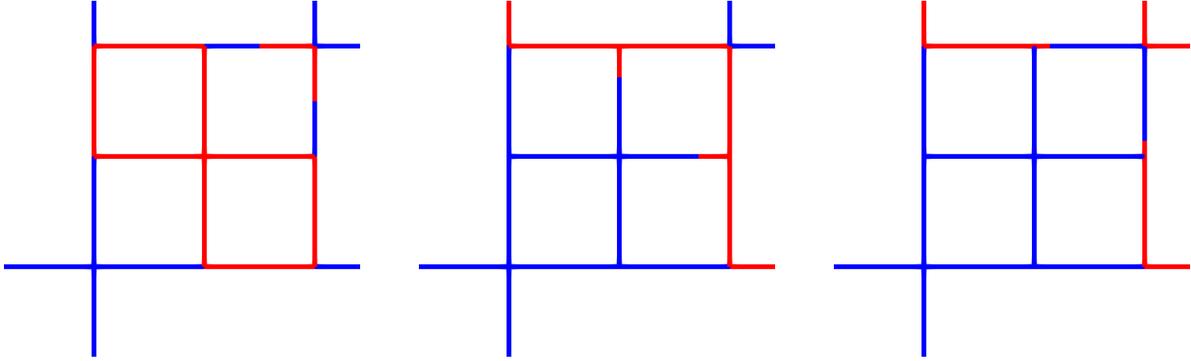

    \centering
    \centering
    \includegraphics[width=0.3\textwidth,trim={16cm 16cm 8cm 8cm},clip]{fig/ex3_non_linear_cond_1}%
    \hspace*{0.05\textwidth}%
    \includegraphics[width=0.3\textwidth,trim={16cm 16cm 8cm 8cm},clip]{fig/ex3_non_linear_cond_3}%
    \hspace*{0.05\textwidth}%
    \includegraphics[width=0.3\textwidth,trim={16cm 16cm 8cm 8cm},clip]{fig/ex3_non_linear_cond_4}
    \caption{Zoom around the short fracture branches $\Omega_s$ of the ``if
    $\Omega_1$'' condition for different iterations, respectively $(1, 3, 4)$,
    for the problem of Section \ref{subsubsec:examples:case3:nonlinear}.}
    \label{fig:ex3_non_linear_zoom}
\end{figure}

This final example shows a very interesting and physically sound final
configuration, which might have been hard to predict without the framework
introduced in this work. All these examples showed the applicability and
importance of the model adaptation and support the presented Algorithm \ref{algo:interface} to
be a valid approach for its solution.

\section{Conclusion}\label{sec:conclusion}

In this work we introduced a new model for discrete fracture networks that is
able to adapt the constitutive relation between velocity and pressure 
depending on the magnitude of the fluid velocity, which is part of the unknowns. 

We presented a mathematical formulation for it, and with an energy argument we were able
to show that under reasonable hypotheses on the constitutive law the problem
has a solution. When the interface inverse permeability jump is non-negative, the problem
is convex and we could show existence in any space dimension; when it is negative, however,
the problem becomes non-convex and we had to restrict our proof of existence to 
one space dimension. We also introduced a discrete algorithm that, for a given problem, tracks the low- and
high-speed regions as well at the interface separating them. We considered various constitutive 
relations for distinct parts of the network, such as the classical Darcy law and the non-linear 
Darcy--Forchheimer law. Several numerical examples showed the validity of 
the proposed approach by increasing the geometrical and physical complexity of the problem. 
We noticed that when the interface inverse permeability jump is positive, the algorithm
seems not to converge and oscillate indefinitely between two configurations. In the
complementary non-positive case, the algorithm seems to behave and converge nicely.
Let us summarize these results:
\begin{table}[!ht]
    \centering
    \begin{threeparttable}
    \begin{tabular}{ |c|c|c|c| }
        \cline{2-4}
        \multicolumn{1}{c|}{} & $\lambda_{21} > 0$ &
        $\lambda_{21} = 0$ &
        $\lambda_{21} < 0$
        \\ \hline
        existence of solutions &  yes & yes & yes if $d=1$\\ \hline
        convexity of energy &  yes  & yes & no\tnote{$_*$}\\ \hline
        convergence of algorithm ($d=1$) & no & yes & yes\\ \hline
    \end{tabular}
        \begin{tablenotes}\footnotesize
        \item[$_*$]\hspace{-1pt}unless in trivial case $d=1$ and $\Sigma_{\mt{v}}\neq\emptyset$
    \end{tablenotes}
\end{threeparttable}
\caption{Summary of results according to the interface inverse permeability jump $\lambda_{21} := \lambda_2-\lambda_1$.}
\end{table}

From a modeling point of view, as mentioned in the introduction, future extensions will
be the inclusion of the rock matrix and the possibility of having more than two 
constitutive laws for the problem.

From an analytical point of view, open questions include the existence of solutions when $d>1$
and the interface inverse permeability jump is negative, and the extension of the existence 
results to tensor permeabilities when $d>1$; see Remark \ref{rem:tensor-perm}. In addition, 
characterizing admissible constitutive laws on the interface (rather than leaving the choice as a free 
parameter of the model, as we did here) as well as determining the Hausdorff dimension of the interface are 
interesting questions that would allow us to study uniqueness of solutions.

From a numerical point of view, further extensions will be the the development of the tracking 
algorithm for $d>1$ and the proof of its convergence
when the interface inverse permeability jump is non-positive, and the development of an alternative
algorithm when the jump is positive.




\bibliographystyle{abbrv}
\bibliography{biblio}

\begin{thebibliography}{10}

\bibitem{Ahmed2018}
E.~Ahmed, A.~Fumagalli, and A.~Budi{\v{s}}a.
\newblock A multiscale flux basis for mortar mixed discretizations of reduced
  {D}arcy--{F}orchheimer fracture models.
\newblock {\em Computer Methods in Applied Mechanics and Engineering},
  354:16--36, 2019.

\bibitem{Ahmed2019}
E.~Ahmed, A.~Fumagalli, A.~Budi{\v{s}}a, E.~Keilegavlen, J.~M. Nordbotten, and
  F.~A. Radu.
\newblock Robust linear domain decomposition schemes for reduced non-linear
  fracture flow models.
\newblock Technical report, arXiv:1906.05831 [math.NA], 2019.

\bibitem{Alboin2000a}
C.~Alboin, J.~Jaffr{\'e}, J.~E. Roberts, X.~Wang, and C.~Serres.
\newblock Domain {D}ecomposition for some {T}ransmission {P}roblems in {F}low
  in {P}orous {M}edia.
\newblock In {\em Numerical treatment of multiphase flows in porous media
  ({B}eijing, 1999)}, volume 552 of {\em Lecture Notes in Phys.}, pages 22--34.
  Springer, Berlin, 2000.

\bibitem{Amir2005}
L.~Amir, M.~Kern, V.~Martin, and J.~E. Roberts.
\newblock D{\'e}composition de domaine et pr{\'e}conditionnement pour un
  mod{\`e}le 3{D} en milieu poreux fractur{\'e}.
\newblock In {\em Proceeding of JANO 8, 8\textsuperscript{th} conference on
  {N}umerical {A}nalysis and {O}ptimization}, Dec. 2005.
\newblock 2005.

\bibitem{Angot2009}
P.~Angot, F.~Boyer, and F.~Hubert.
\newblock Asymptotic and numerical modelling of flows in fractured porous
  media.
\newblock {\em M2AN Mathematical Modelling and Numerical Analysis},
  43(2):239--275, 2009.

\bibitem{AFM18}
J.~Audu, F.~Fairag, and S.~Messaoudi.
\newblock On the well-posedness of generalized {D}arcy--{F}orchheimer equation.
\newblock {\em Boundary Value Problems}, 2018, 2018.

\bibitem{Benedetto2017}
M.~F. Benedetto, A.~Borio, and S.~Scial\`{o}.
\newblock Mixed virtual elements for discrete fracture network simulations.
\newblock {\em Finite Elements in Analysis and Design}, 134:55--67, 2017.

\bibitem{Berre2020a}
I.~Berre, W.~M. Boon, B.~Flemisch, A.~Fumagalli, D.~Gl\"{a}ser, E.~Keilegavlen,
  A.~Scotti, I.~Stefansson, A.~Tatomir, K.~Brenner, S.~Burbulla, P.~Devloo,
  O.~Duran, M.~Favino, J.~Hennicker, I.-H. Lee, K.~Lipnikov, R.~Masson,
  K.~Mosthaf, M.~G.~C. Nestola, C.-F. Ni, K.~Nikitin, P.~Sch\"{a}dle,
  D.~Svyatskiy, R.~Yanbarisov, and P.~Zulian.
\newblock Verification benchmarks for single-phase flow in three-dimensional
  fractured porous media.
\newblock {\em Advances in Water Resources}, 147, 2020.

\bibitem{Berre2019b}
I.~Berre, F.~Doster, and E.~Keilegavlen.
\newblock Flow in fractured porous media: A review of conceptual models and
  discretization approaches.
\newblock {\em Transport in Porous Media}, 130(1):215--236, 2019.

\bibitem{Berrone2013}
S.~Berrone, S.~Pieraccini, and S.~Scial\`{o}.
\newblock {A PDE-constrained optimization formulation for discrete fracture
  network flows}.
\newblock {\em SIAM Journal on Scientific Computing}, 35(2):B487--B510, 2013.

\bibitem{Berrone2013a}
S.~Berrone, S.~Pieraccini, and S.~Scial\`{o}.
\newblock {On simulations of discrete fracture network flows with an
  optimization-based extended finite element method}.
\newblock {\em SIAM Journal on Scientific Computing}, 35(2):908--935, 2013.

\bibitem{Berrone2014}
S.~Berrone, S.~Pieraccini, and S.~Scial\`o.
\newblock An optimization approach for large scale simulations of discrete
  fracture network flows.
\newblock {\em Journal of Computational Physics}, 256(0):838 -- 853, 2014.

\bibitem{Berrone2016}
S.~Berrone, S.~Pieraccini, and S.~Scial\`o.
\newblock Towards effective flow simulations in realistic discrete fracture
  networks.
\newblock {\em Journal of Computational Physics}, 310:181--201, 2016.

\bibitem{Boffi2013}
D.~Boffi, F.~Brezzi, and M.~Fortin.
\newblock {\em Mixed Finite Element Methods and Applications}.
\newblock Springer Series in Computational Mathematics. Springer Berlin
  Heidelberg, 2013.

\bibitem{Boon2020}
W.~M. Boon, J.~M. Nordbotten, and J.~E. Vatne.
\newblock Functional analysis and exterior calculus on mixed-dimensional
  geometries.
\newblock {\em Annali Matematica Pura ed Applicata}, 2020.
\newblock In press.

\bibitem{Boon2018}
W.~M. Boon, J.~M. Nordbotten, and I.~Yotov.
\newblock Robust discretization of flow in fractured porous media.
\newblock {\em SIAM Journal on Numerical Analysis}, 56(4):2203--2233, 2018.

\bibitem{Borio2019a}
A.~Borio, A.~Fumagalli, and S.~Scial\`o.
\newblock Comparison of the response to geometrical complexity of methods for
  unstationary simulations in discrete fracture networks with conforming,
  polygonal, and non-matching grids.
\newblock {\em Computational Geosciences}, 2020.

\bibitem{Brenner2016a}
K.~Brenner, J.~Hennicker, R.~Masson, and P.~Samier.
\newblock {Gradient discretization of hybrid-dimensional {D}arcy flow in
  fractured porous media with discontinuous pressures at matrix-fracture
  interfaces}.
\newblock {\em {IMA Journal of Numerical Analysis}}, Sept. 2016.

\bibitem{BMZ15}
M.~Bulíček, J.~Málek, and J.~Žabenský.
\newblock A generalization of the {D}arcy–{F}orchheimer equation involving an
  implicit, pressure-dependent relation between the drag force and the
  velocity.
\newblock {\em Journal of Mathematical Analysis and Applications},
  424(1):785--801, 2015.

\bibitem{DAngelo2011}
C.~D'Angelo and A.~Scotti.
\newblock A mixed finite element method for {D}arcy flow in fractured porous
  media with non-matching grids.
\newblock {\em Mathematical {M}odelling and {N}umerical {A}nalysis},
  46(02):465--489, 2012.

\bibitem{Dreuzy2013}
J.-R. de~Dreuzy, G.~Pichot, B.~Poirriez, and J.~Erhel.
\newblock Synthetic benchmark for modeling flow in 3d fractured media.
\newblock {\em Computers \& Geosciences}, 50:59 -- 71, 2013.
\newblock Benchmark problems, datasets and methodologies for the computational
  geosciences.

\bibitem{Erhel2009}
J.~Erhel, J.-R. de~Dreuzy, and B.~Poirriez.
\newblock Flow simulation in three-dimensional discrete fracture networks.
\newblock {\em SIAM Journal on Scientific Computing}, 31(4):2688--2705, 2009.

\bibitem{Facciola2019}
C.~Facciol\`a, P.~F. Antonietti, and M.~Verani.
\newblock Mixed-primal discontinuous galerkin approximation of flows in
  fractured porous media on polygonal and polyhedral grids.
\newblock {\em PAMM}, 19(1):e201900117, 2019.

\bibitem{Flemisch2016a}
B.~Flemisch, I.~Berre, W.~Boon, A.~Fumagalli, N.~Schwenck, A.~Scotti,
  I.~Stefansson, and A.~Tatomir.
\newblock Benchmarks for single-phase flow in fractured porous media.
\newblock {\em Advances in Water Resources}, 111:239--258, January 2018.

\bibitem{Formaggia2012}
L.~Formaggia, A.~Fumagalli, A.~Scotti, and P.~Ruffo.
\newblock A reduced model for {D}arcy's problem in networks of fractures.
\newblock {\em {ESAIM}: {M}athematical {M}odelling and {N}umerical {A}nalysis},
  48:1089--1116, 7 2014.

\bibitem{Frih2011}
N.~Frih, V.~Martin, J.~E. Roberts, and A.~Sa{\^a}da.
\newblock Modeling fractures as interfaces with nonmatching grids.
\newblock {\em Computational Geosciences}, 16(4):1043--1060, 2012.

\bibitem{Frih2008}
N.~Frih, J.~E. Roberts, and A.~Saada.
\newblock Modeling fractures as interfaces: a model for {F}orchheimer
  fractures.
\newblock {\em Computers and Geosciences}, 12(1):91--104, 2008.

\bibitem{Fumagalli2016a}
A.~Fumagalli and E.~Keilegavlen.
\newblock Dual virtual element method for discrete fractures networks.
\newblock {\em SIAM Journal on Scientific Computing}, 40(1):B228--B258, 2018.

\bibitem{Fumagalli2017a}
A.~Fumagalli and E.~Keilegavlen.
\newblock Dual virtual element methods for discrete fracture matrix models.
\newblock {\em Oil \& Gas Science and Technology - Revue d'IFP Energies
  nouvelles}, 74(41):1--17, 2019.

\bibitem{Scialo2017}
A.~Fumagalli, E.~Keilegavlen, and S.~Scial\`o.
\newblock Conforming, non-conforming and non-matching discretization couplings
  in discrete fracture network simulations.
\newblock {\em Journal of Computational Physics}, 376:694--712, 2019.

\bibitem{Geuzaine2009}
C.~Geuzaine and J.-F. Remacle.
\newblock Gmsh: A 3-d finite element mesh generator with built-in pre- and
  post-processing facilities.
\newblock {\em International Journal for Numerical Methods in Engineering},
  79(11):1309--1331, 2009.

\bibitem{Keilegavlen2020}
E.~Keilegavlen, R.~Berge, A.~Fumagalli, M.~Starnoni, I.~Stefansson, J.~Varela,
  and I.~Berre.
\newblock Porepy: An open-source software for simulation of multiphysics
  processes in fractured porous media.
\newblock {\em Computational Geosciences}, 2020.

\bibitem{Knabner2014}
P.~Knabner and J.~E. Roberts.
\newblock Mathematical analysis of a discrete fracture model coupling {D}arcy
  flow in the matrix with {D}arcy-{F}orchheimer flow in the fracture.
\newblock {\em ESAIM: Mathematical Modelling and Numerical Analysis},
  48:1451--1472, 9 2014.

\bibitem{Martin2005}
V.~Martin, J.~Jaffr{\'e}, and J.~E. Roberts.
\newblock Modeling {F}ractures and {B}arriers as {I}nterfaces for {F}low in
  {P}orous {M}edia.
\newblock {\em SIAM Journal on Scientific Computing}, 26(5):1667--1691, 2005.

\bibitem{Morales2017}
F.~A. Morales and R.~E. Showalter.
\newblock A {D}arcy-brinkman model of fractures in porous media.
\newblock {\em Journal of Mathematical Analysis and Applications}, 452(2):1332
  -- 1358, 2017.

\bibitem{Nordbotten2018}
J.~M. Nordbotten, W.~Boon, A.~Fumagalli, and E.~Keilegavlen.
\newblock Unified approach to discretization of flow in fractured porous media.
\newblock {\em Computational Geosciences}, 23(2):225--237, 2019.

\bibitem{Raviart1977}
P.-A. Raviart and J.-M. Thomas.
\newblock A mixed finite element method for second order elliptic problems.
\newblock {\em Lecture Notes in Mathematics}, 606:292--315, 1977.

\bibitem{Roberts1991}
J.~E. Roberts and J.-M. Thomas.
\newblock Mixed and hybrid methods.
\newblock In {\em Handbook of numerical analysis, {V}ol.\ {II}}, Handb. Numer.
  Anal., II, pages 523--639. North-Holland, Amsterdam, 1991.

\bibitem{Rybak2020}
I.~Rybak and S.~Metzger.
\newblock A dimensionally reduced stokes-{D}arcy model for fluid flow in
  fractured porous media.
\newblock {\em Applied Mathematics and Computation}, 384:125260, 2020.

\bibitem{SLM13}
J.~J. Salas, H.~López, and B.~Molina.
\newblock An analysis of a mixed finite element method for a
  {D}arcy–{F}orchheimer model.
\newblock {\em Mathematical and Computer Modelling}, 57(9):2325--2338, 2013.
\newblock System Dynamics in Project Management \& Applied Mathematics and
  Computational Science and Engineering—Selected Papers of the Seventh
  PanAmerican Workshop — June 6–11 2010, Venezuela.

\bibitem{Schwenck2015}
N.~Schwenck, B.~Flemisch, R.~Helmig, and B.~Wohlmuth.
\newblock Dimensionally reduced flow models in fractured porous media:
  crossings and boundaries.
\newblock {\em Computational Geosciences}, 19(6):1219--1230, 2015.

\bibitem{SV01}
F.~Spena and A.~Vacca.
\newblock A potential formulation of non-linear models of flow through
  anisotropic porous media.
\newblock {\em Transport in Porous Media}, 45:405--421, 12 2001.

\bibitem{Zeng2006}
Z.~Zeng and R.~Grigg.
\newblock A criterion for non-{D}arcy flow in porous media.
\newblock {\em Transport in Porous Media}, 63(1):57--69, 2006.

\end{thebibliography}

\end{document}